\documentclass[11pt,reqno]{amsart}

\usepackage{amsmath, amsfonts, amsthm, amssymb,graphicx, color, cite,enumitem}
\usepackage[margin=1.2in]{geometry}

\usepackage{latexsym}

\usepackage{hyperref}
\hypersetup{hidelinks}

\usepackage{mathtools}

\newtheorem{Theorem}{Theorem}[section]
\newtheorem{Lemma}{Lemma}[section]
\newtheorem{Proposition}{Proposition}[section]

\theoremstyle{definition}

\theoremstyle{remark}
\newtheorem{Remark}{Remark}[section]

\numberwithin{equation}{section}

\allowdisplaybreaks

\def\im{\mathrm{i}\,}

\def\eps{\varepsilon}
\def\e{\varepsilon}
\def\epsilon{\varepsilon}
\def\de{{\partial}}

\def\drm{{\mathrm{d}}}

\def\uno{{\mathrm{Id}}}

\def\RR {\mathbb{R}}
\def\CC {\mathbb{C}}
\def\TT {\mathbb{T}}
\def\ZZ {\mathbb{Z}}
\def\NN {\mathbb{N}}

\def\Re{{\rm Re}}

\newcommand{\cA}{\mathcal{A}}
\newcommand{\cB}{\mathcal{B}}

\newcommand{\cD}{\mathcal{D}}

\newcommand{\cI}{\mathcal{I}}

\newcommand{\cK}{\mathcal{K}}
\newcommand{\cL}{\mathcal{L}}
\newcommand{\cM}{\mathcal{M}}

\newcommand{\cO}{\mathcal{O}}

\newcommand{\cR}{\mathcal{R}}

\newcommand{\cT}{\mathcal{T}}
\newcommand{\cU}{\mathcal{U}}
\newcommand{\cV}{\mathcal{V}}

\author[R.M. Chen]{Robin Ming Chen}
\address{Department of Mathematics, University of Pittsburgh, Pittsburgh, PA 15260, USA.}
\email{mingchen@pitt.edu}

\author[T. Dai]{Tian Dai}
\address{Department of Mathematics, University of Pittsburgh, Pittsburgh, PA 15260, USA.}
\email{tid34@pitt.edu}

\author[D. Wang]{Dehua Wang}
\address{Department of Mathematics, University of Pittsburgh, Pittsburgh, PA 15260, USA.}
\email{dhwang@pitt.edu}

\author[W. Wang]{Weiqiang Wang}
\address{Department of Mathematics, University of Pittsburgh, Pittsburgh, PA 15260, USA.}
\email{wew179@pitt.edu}

\title[Long-wave shear instability on the $\beta$-plane]
{Long-wave Stability and Instability of Periodic Shear Flows for the 2D Navier-Stokes Equations on the $\beta$-Plane}

\date{\today}

\begin{document}
\begin{abstract}

We study the spectral stability of periodic shear flows for the two-dimensional Navier--Stokes equations on the $\beta$-plane in the long-wave regime. While it is known that non-rotating periodic shear flows are generically unstable to sufficiently long-wave perturbations, our results show that planetary rotation suppresses this instability mechanism. Using a perturbative analysis based on Kato's reduction, we derive an asymptotic expansion for the principal eigenvalue of the linearized operator that is uniform in the Rossby parameter $\beta \geq 0$, and establish an explicit stability criterion in terms of the shear profile, the viscosity, and the Rossby parameter. Under the critical scaling where the Rossby parameter $\beta$ and the rescaled longitudinal wavenumber $\e$ are of comparable size, this criterion extends Yudovich's classical long-wave instability threshold to rotating flows and reveals a sharp transition between stability and instability governed by the range of the ratio $\beta/\e$. These results provide a rigorous analysis of how viscosity, shear, and planetary rotation interact to determine long-wave stability on the $\beta$-plane.

\end{abstract}

\maketitle

\section{Introduction}\label{sec intro}
We are concerned with the stability and instability of shear flows for the two-dimensional Navier--Stokes equations on the $\beta$-plane with periodic boundary conditions. The 
Navier--Stokes equations describe the evolution of incompressible viscous fluids and have been extensively studied from both physical and mathematical perspectives. In geophysical fluid dynamics, planetary rotation plays a fundamental role in determining large-scale oceanic and atmospheric motions (see, e.g., \cite{Gallagher}). A classical approximation for planetary rotation flows is the $\beta$-plane model, in which the Coriolis parameter varies linearly with latitude (see, e.g., \cite{Ped87,Val17}). This approximation replaces the spherical geometry of the Earth with a flat plane while retaining the meridional variation of the Coriolis force in the following form, which is crucial for describing large-scale phenomena such as Rossby waves:
\begin{equation}\label{Eq-NS-beta}
\begin{cases}
\partial_{t} w + \mathbf{v} \cdot \nabla w + \beta v_{2}
= \nu \Delta w + f,
\quad (\tilde{x},y) \in \mathbb{T}_{\alpha} \times \mathbb{T}, \\[4pt]
\mathbf{v} = (v_1,v_2) = (\phi_{y},-\phi_{\tilde{x}}),
\quad w = -\Delta \phi,
\end{cases}
\end{equation}
where $\mathbf{v}$ denotes the velocity field, $w$ the vorticity, $f$ the external forcing term, $\nu>0$ the kinematic viscosity, $\phi$ the stream function, and $\beta\geq0$ the Rossby parameter. The value $\beta=0$ corresponds to the non-rotating case treated in \cite{Colombo}. The domain is the rectangular torus
\[
\mathbb{T}_{\alpha} \times \mathbb{T} := \mathbb{R}/\left(\tfrac{2\pi}{\alpha}\mathbb{Z} \right) \times \mathbb{R}/(2\pi\mathbb{Z})
\simeq \left[0,\tfrac{2\pi}{\alpha} \right) \times [0, 2\pi),
\]
where $0 < \alpha<1$ represents the inverse aspect ratio. We use the notation
\[
\nabla = (\partial_{\tilde{x}},\partial_y),
\qquad
\Delta = \partial_{\tilde{x}\tilde{x}} + \partial_{yy}.
\]
Our goal is to investigate the stability of shear flows of the form
\[
\mathbf{v}_E = (U(y),0),
\qquad
w_E = -U'(y),
\]
where $U$ is sufficiently smooth and has zero mean:
\[
\int_{\mathbb{T}} U(y)\,dy = 0.
\]
These profiles are steady states of \eqref{Eq-NS-beta}, provided that the external forcing $f$ is chosen as
\[
f(\tilde{x},y) = \nu U'''(y).
\]

\subsection{Historical research and comparison with previous literature}

We first review several developments that motivate the present work.

\subsection*{The stability and instability of shear flows} The stability theory of shear flows has a long history dating back to the classical works of Helmholtz, Kelvin, Rayleigh, Orr, Sommerfeld, and Tollmien; see the monographs \cite{Chandrasekhar,Drazin} for the classical theory. In the modern PDE framework, substantial progress has been achieved during the past two decades. In particular, the works of Bedrossian and Masmoudi \cite{Bedrossian2015} and of Bedrossian, Masmoudi and Vicol \cite{Bedrossian2016} established the mechanisms of enhanced dissipation and inviscid damping, leading to a quantitative nonlinear stability theory near Couette flow.

For the Navier-Stokes equations, the stability of shear flows depends strongly on the underlying boundary conditions. In domains with physical boundaries and no-slip conditions, boundary layers may generate strong instability mechanisms \cite{Bian,Grenier16,Grenier16Adv}. In contrast, within the periodic framework considered here, Kolmogorov \cite{Arnold} 
initiated the study of perturbations around the shear flow $U(y)=\sin(y)$ for different aspect ratios $\alpha$ and viscosities $\nu$. This flow, now known as the Kolmogorov flow, was first analyzed by Meshalkin and Sinai \cite{Mehsalkin61Investigation}, who used continued-fraction methods to prove spectral stability for $\alpha>1$ and instability for sufficiently small $\alpha$, depending on the viscosity parameter.

In 1966, Yudovich \cite{Yudovich} identified a long-wave instability mechanism for two-dimensional Navier--Stokes shear flows. He derived the formal criterion that long-wave perturbations become unstable precisely when
\begin{equation}\label{Yudovich}
  \|\partial^{-1}_{y} U\|_{L^2} > \nu.
\end{equation}
Although the original argument was formal and lacked a complete rigorous justification, it provided an important prediction for the instability threshold.

Motivated by Yudovich's observation, Colombo, Dolce, Montalto, and Ventura \cite{Colombo} recently established the first rigorous proof of long-wave instability for general periodic shear flows in elongated tori $\mathbb{T}_{\alpha}\times\mathbb{T}$ with $\alpha\ll1$. Their result confirmed Yudovich's criterion and extended the Meshalkin--Sinai instability mechanism beyond the Kolmogorov flow. The authors employed two complementary approaches, based on Kato's perturbation theory \cite{Kato} and a normal-form reduction, to overcome the singular perturbation arising in the linearized operator as $\alpha\rightarrow0$. Their analysis showed that the instability is genuinely associated with long-wave perturbations and is fundamentally different from boundary-layer instabilities occurring in no-slip domains. These results are stated for the classical, non-rotating two-dimensional Navier--Stokes equations. We discuss their relation to rotating flows below.

\subsection*{The effect of rotation}
The stability of fluid flows under rotational effects has been studied extensively in both geophysical and astrophysical contexts. This includes phenomena like the Taylor--Couette instability, centrifugal stability criteria, and Rossby waves caused by planetary rotation. A full review of this literature is beyond this paper's scope; for more details, see \cite{Gallagher} on Earth's rotation effects on geophysical flows and \cite{Boffetta} for a survey of two-dimensional turbulence. Here, we focus only on the works most relevant to our current setting.

A classical model incorporating rotational effects is the Navier--Stokes--Coriolis system \cite{Chandrasekhar}:
\begin{equation}\label{NSC}
\begin{cases}
  \partial_t \mathbf{v} + (\mathbf{v} \cdot \nabla) \mathbf{v}  + \nabla p  + \boldsymbol{\Omega} \times \mathbf{v} = \nu \Delta \mathbf{v} + \mathbf{f}, \\
  \nabla \cdot \mathbf{v} = 0,
\end{cases}
\end{equation}
where $\mathbf{v}$ is the velocity field, $p$ the pressure, $\nu>0$ the kinematic viscosity, $\boldsymbol{\Omega}$ the angular velocity vector, and $\mathbf{f}$ the external forcing.

It is worth noting that a \emph{constant} background rotation has no effect on 
two-dimensional vorticity dynamics. Indeed, for a two-dimensional flow $\mathbf{v}=(v_1,v_2,0)$ and $\boldsymbol{\Omega}=(0,0,\beta)=:\beta\hat{\mathbf z}$ with $\beta$ constant, the divergence-free condition gives $\nabla\times(\boldsymbol{\Omega}\times\mathbf v)=\beta(\nabla\cdot\mathbf v)\hat{\mathbf z}=0$, so taking the curl of the momentum equation in \eqref{NSC} gives exactly the vorticity equation for $\beta=0$.
Consequently, the long-wave instability results of \cite{Colombo}, although
formulated for the non-rotating equations, carry over without modification to
the two-dimensional Navier--Stokes--Coriolis system with constant Coriolis
parameter. The $\beta$-plane equation \eqref{Eq-NS-beta} is the special case
$\boldsymbol{\Omega}=(0,0,\beta y)$ of \eqref{NSC} written in vorticity form,
and it is precisely the $y$-dependence of $\boldsymbol{\Omega}$ that yields
the term $\beta v_2$ in \eqref{Eq-NS-beta} and that the analysis of
\cite{Colombo} does not cover.

Motivated by the seminal work of Bedrossian, Germain and Masmoudi \cite{Bedrossian2017}, a number of recent papers \cite{Fan,HuangW202409,HuangW202412,Li,Zelati} have investigated transition thresholds near the Couette flow for \eqref{NSC} on $\TT \times \RR \times \TT$ with constant rotation $\boldsymbol{\Omega}=(0,0,\beta)$ and $\mathbf{f}=\mathbf{0}$; throughout this paragraph $\beta$ denotes the constant rotation rate of that three-dimensional problem, rather than the Rossby parameter of \eqref{Eq-NS-beta}. In general the Coriolis force is stabilizing, but the strength of the effect depends sensitively on $\beta$. The key quantity is the Bradshaw--Richardson number $B_{\beta}:=\beta(\beta-1)$, which distinguishes the stable from the unstable regimes; see, e.g., \cite{Bradshaw,HuangY}. At high Reynolds number $\mathrm{Re}:=\nu^{-1}$, the transition threshold is measured by the exponent $\gamma$ for which initial perturbations of size $\|u_{\mathrm{in}}\|_{H^\sigma}\lesssim \mathrm{Re}^{-\gamma}$ remain stable, so that a smaller $\gamma$ corresponds to a stronger stability result. Without rotation, the sharp threshold $\gamma\leq3/2$ was established in \cite{Bedrossian2017}. Huang, Sun and Xu \cite{HuangW202409,HuangW202412} showed that a dispersive structure in the zero-frequency modes suppresses the lift-up effect, improving the threshold to $\gamma\leq1$ for $B_\beta>0$ under the additional assumption $|\beta|\geq2$, and to $\gamma\leq2$ at the degenerate point $\beta=1$, where $B_\beta=0$. Coti Zelati, Del Zotto and Widmayer \cite{Zelati} gave a more complete picture of the linearized dynamics, identifying three regimes, namely lift-up stability when $B_\beta=0$, exponential instability when $B_\beta<0$, and dispersive behaviour when $B_\beta>0$, and refined the threshold to $\gamma\leq8/9$ for $B_\beta>0$ and to $\gamma\leq5/6$ for $B_\beta\gtrsim\nu^{-1}$, the latter under the additional assumption that the initial double zero modes vanish. More recently, Li, Sun, Wang, Wei and Zhang \cite{Li} removed this restriction. Together, these works show that rotation can substantially enhance nonlinear stability, and that the improvement becomes more pronounced as the underlying dispersive mechanism is better understood.

On the other hand, Fan, Han and Wang \cite{Fan} considered the regime $\beta \in \big(\frac{2}{17}(5-2\sqrt2),\frac{2}{17}(5+2\sqrt2)\big) \subset(0,1)$ in three dimensions and proved nonlinear instability in the sense of Hadamard. Their result shows that the stabilizing effects of inviscid damping and enhanced dissipation may fail to overcome the linear instability generated by the Coriolis force in certain regimes.

 For the inviscid problem, classical necessary conditions for instability include the Rayleigh--Kuo criterion \cite{Kuo,Rayleigh}, which extends Rayleigh's inflection-point criterion by incorporating the variation of planetary vorticity, and generalizations of Howard's semicircle theorem \cite{Pedlosky1963}; further characterizations of the unstable wave speeds may be found in \cite{Kuo,Miles,Pedlosky1964}. Recently, Lin, Yang and Zhu \cite{Lin} studied the barotropic instability of a class of shear flows on the $\beta$-plane and developed a general approach for determining sharp stability conditions. In particular, their detailed analysis of the Sinus flow exhibits an intricate interplay between the wavenumber and the Rossby parameter $\beta$, showing that the dynamics near rotating shear flows is considerably richer than in the non-rotating case.

\subsection*{Comparison of our results with the previous literature}
%

Despite the extensive literature on rotating flows with constant Coriolis parameters, mathematically rigorous results for viscous $\beta$-plane flows remain relatively limited. A related study is the work of Al-Jaboori and Wirosoetisno \cite{Al-Jaboori}, who proved that, for sufficiently regular forcing and sufficiently large $\beta$, solutions become asymptotically zonal, and that under an additional smallness condition the global attractor reduces to a single point. Their results show that rotation drives the flow toward zonal states.
In this paper, we study whether such a zonal state is stable, and at which wavenumbers.

To the best of our knowledge, no previous work has established a spectral stability theory for periodic shear flows of the viscous $\beta$-plane equations.
Compared with the constant-rotation case, the $\beta$-term breaks the symmetry in the $y$-direction and leads to a more intricate interaction between shear, viscosity, and rotation. In particular, the long-wave spectral structure becomes significantly more delicate.
In this paper, we address this problem in the long-wave regime.
Specifically, we derive an asymptotic expansion for the principal eigenvalue of
the linearized operator in the long-wave parameter $\eps$.
The expansion is uniform in $\beta \geq 0$, with no lower bound on $\beta$ imposed at any stage, so that it reduces at $\beta=0$ to the expansion of \cite{Colombo}. For $\beta>0$ it replaces Yudovich's threshold \eqref{Yudovich} by a one-parameter
family, across which a sharp transition between long-wave stability and instability occurs. Precise statements are given in the next subsection, and the difficulties that the $\beta$-term introduces into the analysis are described in Subsection \ref{ssec:difficulties}.

\subsection{Main results and strategy}

We linearize \eqref{Eq-NS-beta} around a shear flow and consider perturbations of the form $\omega(t,\alpha\tilde{x},y)$. After introducing the rescaled variable $x=\alpha\tilde{x}$, the linearized equation becomes
\begin{equation}
\begin{cases}
\partial_t \omega
+ \alpha U(y)\partial_x \omega
+ \alpha (U''(y)-\beta)\partial_x \psi
= \nu \Delta_\alpha \omega,
\quad (x,y)\in \mathbb{T}^2, \\[4pt]
\mathbf{u} = (\psi_y,-\psi_x),
\qquad
\omega = -\Delta_\alpha \psi,
\end{cases}
\end{equation}
where $\Delta_\alpha := \partial_{yy}+\alpha^2\partial_{xx}$. The corresponding linear operator is
\begin{equation}
\boldsymbol{\mathcal L}_{\nu,\alpha,\beta}
:= \nu \Delta_\alpha
- \alpha \partial_x \Big(
U(y)
- (U''(y)-\beta)\Delta_\alpha^{-1}
\Big).
\end{equation}

Our analysis is motivated by the works \cite{Colombo,Yudovich}, where long-wave spectral instability was established for non-rotating shear flows ($\beta=0$). In particular, under the condition \eqref{Yudovich} and for sufficiently long waves satisfying
\[
\alpha |k| < \delta_0 \nu,
\]
they showed that the linearized operator $\boldsymbol{\mathcal L}_{\nu,\alpha,0}$ possesses an eigenvalue with positive real part.

The main observation of the present work is that the Coriolis effect fundamentally modifies this instability mechanism. In the long-wave regime, the $\beta$-term introduces a stabilizing contribution, and the competition between shear-induced instability, viscous damping, and rotation determines the spectral behavior.

Since $\boldsymbol{\mathcal L}_{\nu,\alpha,\beta}$ preserves the Fourier modes in the $x$-direction, we take the Fourier transform in $x$ and reduce the problem to the one-dimensional operator
\begin{equation}\label{Operator-L}
\mathcal L_{\nu,\alpha k,\beta}
:= \nu(\partial_{yy}-\alpha^2k^2)
- \im \alpha k\Big(
U(y)
- (U''(y)-\beta)
(\partial_{yy}-\alpha^2k^2)^{-1}
\Big),
\end{equation}
acting from $H^2(\mathbb T)$ to $L^2(\mathbb T)$. 

To investigate the long-wave regime, we introduce
\[
\varepsilon := \alpha |k|,
\]
and without loss of generality we assume $k>0$. Indeed, $U$ and $U''-\beta$ are real-valued, so that $\mathcal L_{\nu,-\e,\beta}=\overline{\mathcal L_{\nu,\e,\beta}}$, and the two spectra are complex conjugates of one another, leaving $\Re\,\lambda$ unchanged. We abbreviate $\mathcal L_{\nu,\e,\beta}\coloneqq\mathcal L_{\nu,\alpha k,\beta}$.

We analyze the spectrum of $\mathcal L_{\nu,\varepsilon,\beta}$ using perturbation theory in $\varepsilon$. Since $(\partial_{yy}-\varepsilon^2)^{-1}$ becomes singular at the zero Fourier mode as $\varepsilon\to0$, producing in \eqref{Operator-L} a contribution of size $\varepsilon^{-1}$ carried entirely by that mode, we isolate the zero mode and decompose
\begin{equation}\label{decomposition L}
\mathcal L_{\nu,\varepsilon,\beta} = \mathcal M_{\nu,\varepsilon,\beta} - \im \varepsilon \mathcal R_{\varepsilon,\beta},
\end{equation}
where
\begin{equation}\label{Operator-M}
\mathcal M_{\nu,\varepsilon,\beta} := \nu \mathcal D_\varepsilon - \frac{\im}{\varepsilon}(U''(y)-\beta)\Pi_0,
\end{equation}
\begin{equation}\label{Operator-R}
\mathcal R_{\varepsilon,\beta} := U(y) + (U''(y)-\beta)(-\partial_{yy}+\varepsilon^2)^{-1}\Pi_{\neq}.
\end{equation}
Here $\mathcal D_\varepsilon := \partial_{yy}-\varepsilon^2$, and $\Pi_0$, $\Pi_{\neq}$ denote the orthogonal projections onto the zero and nonzero Fourier modes in $y$, respectively. We regard $\mathcal M_{\nu,\varepsilon,\beta}$ as the leading-order operator and $\mathcal R_{\varepsilon,\beta}$ as a perturbation.

\subsection*{Notation.}  All function spaces are complex-valued unless stated otherwise, thus $L^2(\TT):=L^2(\TT;\CC)$, and similarly for the other spaces. We equip $L^2(\TT)$ with the inner product
\begin{equation}\label{innerproduct}
  \langle f, g \rangle_{L^2} = \frac{1}{2\pi} \int_{0}^{2\pi} \overline{f(y)} g(y) \drm y, \quad \forall f,g \in L^2(\TT),
\end{equation}
which is conjugate-linear in the first argument and linear in the second. With this convention, the orthogonal projection onto the zero Fourier mode is the linear functional
\[\Pi_0 (f) = \langle 1, f \rangle_{L^2}, \quad \forall f \in L^2(\TT),\]
and $\Pi_{\neq}:=\uno-\Pi_0$ is the projection onto the nonzero modes. We write
\[
  f_\ell \coloneqq \langle e^{\im \ell y}, f\rangle_{L^2} = \frac{1}{2\pi}\int_0^{2\pi} f(y) e^{-\im \ell y}\drm y , \qquad \ell \in \ZZ,
\]
for the Fourier coefficients of $f$, so that $$f=\sum_{\ell\in\ZZ}f_\ell e^{\im \ell y} \quad \text{and} \quad \|f\|_{L^2}^2=\sum_{\ell\in\ZZ}|f_\ell|^2.$$ Since $U$ is real-valued with zero average, $U_0=0$ and $U_{-\ell}=\overline{U_\ell}$. We use the norms
\[
\|f\|_{L^2} \coloneqq \langle f,f\rangle_{L^2}^{1/2},
\qquad
\|f\|_{L^\infty} \coloneqq \sup_{y\in\TT}|f(y)| ,
\qquad
\|f\|_{C^2} \coloneqq \max\big\{\|f\|_{L^\infty},\|f'\|_{L^\infty},\|f''\|_{L^\infty}\big\} .
\]
For a bounded linear operator $T:X\to Y$, we write $\|T\|_{\cB(X,Y)}\coloneqq\sup_{\|f\|_X=1}\|Tf\|_Y$. Linear functionals on $L^2(\TT)$ are complex-valued, so they are elements of $\cB(L^2,\CC)$. For a densely defined operator $A$ on $L^2(\TT)$, we denote by $\sigma_{L^2}(A)$ and $\rho_{L^2}(A)$ its spectrum and resolvent set. We adopt the convention $0\in\NN$, so that $\NN=\{0,1,2,\dots\}$ and $\ZZ^+=\{1,2,\dots\}$.

Finally, we write $a\lesssim b$, or $a=\cO(b)$, if $a\leq C_U b$ for a constant $C_U>0$ depending only on $\|U\|_{C^2}$, and independent of $\e$, $\nu$, $\beta$ and $j$. For operators and functions we write, respectively,
\begin{equation}\label{def-O-notation}
  T = \cO_{\rm op}(\kappa) \iff \|T\|_{\cB(L^2,L^2)} \leq C_U\kappa ,
  \qquad
  f = \cO_{\rm fun}(\kappa) \iff \|f\|_{L^2}\leq C_U\kappa .
\end{equation}
Throughout the paper, we denote
\begin{equation}\label{def-b}
b\coloneqq \frac{\beta}{\e\nu}.
\end{equation}


Our main result is stated as follows.
\begin{Theorem}\label{thm:main-1}
Let $U\in C^3(\TT)$ be a mean-zero shear profile, and let $0<\nu\leq 1$. Then there exists a constant $\delta_{U} \in (0,1)$, depending only on $\|U\|_{C^2}$, such that for every $\beta \geq 0$ and every $\e \in (0,1)$ satisfying
\begin{equation}\label{keyassumption}
\frac{\e}{\nu}(1+\beta) < \delta_{U} ,
\end{equation}
the following hold:
\begin{enumerate}[label=\emph{(\roman*)},leftmargin=*]
  \item \emph{(Spectrum)} The spectrum of $\cL_{\nu,\e,\beta}$ consists of isolated eigenvalues of finite multiplicity, namely a simple eigenvalue $\lambda_{\nu,\e,\beta}^{(0)}$ together with, for each $j\in\ZZ^+$, either two eigenvalues or one eigenvalue of algebraic multiplicity two, denoted $\lambda_{\nu,\e,\beta}^{(\pm j)}$, and
  \begin{equation}\label{thm-damped}
    \big|\lambda_{\nu,\e,\beta}^{(\pm j)} + \nu(j^2+\e^2)\big| < \frac{\nu}{2},
    \qquad j \geq 1 .
  \end{equation}
  \item \emph{(Expansion of the principal eigenvalue)} The eigenvalue $\lambda^{(0)}_{\nu,\e,\beta}$ satisfies
\begin{equation}\label{eigenvalue-beta=nueps}
  \Re \, \lambda_{\nu,\e,\beta}^{(0)} = \frac{\e^2}{\nu}\left[ -\nu^2  - \left \langle U, \left( \partial_{y}^4 + \frac{\beta^2}{\e^2 \nu^2} \right)^{-1} U'' \right\rangle_{L^2} + \cO \left( \theta_{\nu,\e,\beta}\,\frac{\e}{\nu}(1+\beta)^3 \right) \right],
\end{equation}
where
\begin{equation}\label{def-theta}
  \theta_{\nu,\e,\beta} \coloneqq \min\Big\{ 1, \frac{\e\nu}{\beta} \Big\} = \min\{1,b^{-1}\} \in (0,1], \text{ for } \beta>0, \quad \text{and}\quad  \theta_{\nu,\e,0}\coloneqq 1 .
\end{equation}
\end{enumerate}
\end{Theorem}


A few remarks and observations are presented below.

\begin{Remark}\label{rem1}
We may write
\begin{equation}\label{rep-Dbeta}
  \left\langle U, \left( \partial_{y}^4 + \frac{\beta^2}{\e^2 \nu^2} \right)^{-1} U'' \right\rangle_{L^2} = - \left\| \left( \partial_{yy} - \frac{\im \beta}{\e \nu} \right)^{-1} U' \right\|_{L^2}^2 = -\sum_{j \neq 0} \frac{j^2|U_j|^2}{j^4+b^2} ,
\end{equation}
so that, setting
\begin{equation}\label{def-GU}
  G_U(b) \coloneqq \left\| \left( \partial_{yy} - \im b \right)^{-1} U' \right\|_{L^2} = \left( \sum_{j \neq 0} \frac{j^2 |U_j|^2}{j^4+b^2} \right)^{1/2} , \qquad b \geq 0 ,
\end{equation}
the expansion \eqref{eigenvalue-beta=nueps} takes the compact form
\begin{equation}\label{eigenvalue-GU}
  \Re \, \lambda_{\nu,\e,\beta}^{(0)} = \frac{\e^2}{\nu}\left[ G_U(b)^2 - \nu^2 + \cO \left( \theta_{\nu,\e,\beta}\,\frac{\e}{\nu}(1+\beta)^3 \right) \right].
\end{equation}
When $\beta=0$, we have $b=0$, $\theta_{\nu,\e,0}=1$ and $G_U(0)=\|\partial_y^{-1}U\|_{L^2}$, so \eqref{eigenvalue-GU} is exactly the expansion of \cite[Theorem 2.1]{Colombo} and recovers Yudovich's criterion \eqref{Yudovich}. Since the smallness condition \eqref{keyassumption} and the remainder are uniform in $\beta \geq 0$, the rotating expansion degenerates continuously to the non-rotating one as $\beta\to0$.

The scalar function $ G_U$ governs the full stability picture. From \eqref{def-GU}, $G_U(b)$ is continuous and strictly decreasing on $[0,\infty)$, with the upper bound $G_U(0)=\|\de_y^{-1}U\|_{L^2}$ and $G_U(b)\to0$ as $b\to\infty$. Several observations are discussed below.

(1) \emph{Suppression at a fixed rotation rate}. For fixed $\beta$ and $\nu$, the real part of $\lambda_{\nu,\e,\beta}^{(0)}$ remains negative for sufficiently small $\e>0$, in sharp contrast with the non-rotating case where long-wave instability occurs under the condition \eqref{Yudovich}. Thus, the planetary rotation can suppress the long-wave instability.

(2) \emph{Transition threshold between long-wave stability and instability}.
For fixed $\nu$, we may characterize the stable and unstable regimes in the $(\beta,\e)$ plane. From \eqref{eigenvalue-GU}, $G_U(b)^2$ has the size $\nu^{2}$ exactly when $b\approx \nu^{-1}$, that is $\beta\approx\e$. This motivates the critical scaling
\begin{equation}\label{scaling-beta=Ke}
  \beta = K \e, \qquad K \geq 0 \ \text{ fixed}, \qquad\text{or equivalently}\qquad b = \frac{K}{\nu}.
\end{equation}
Under this scaling, the expansion \eqref{eigenvalue-GU} becomes
\begin{equation}\label{eigenvalue-K}
  \Re \, \lambda_{\nu,\e,\beta}^{(0)} = \frac{\e^2}{\nu}\left[ G_U \left(\frac{K}{\nu}\right)^2 - \nu^2 + \cO \left( \theta_{\nu,\e,K\e}\,\frac{\e}{\nu}(1+K\e)^3 \right) \right].
\end{equation}
If Yudovich's condition \eqref{Yudovich} fails, then $G_U \left(\frac{K}{\nu}\right) < G_U(0) \leq \nu$ for any $K>0$ and $\nu>0$. Hence, under the critical scaling $\beta=K\e$, the flow is long-wave stable as $\e\to0$, $\beta=K\e\to0$. On the other hand, if Yudovich's condition \eqref{Yudovich} holds, then there exists a unique transition threshold $K_\ast=K_\ast(U,\nu)>0$ such that $G_U \left(\frac{K_\ast}{\nu}\right)=\nu$. This threshold separates the long-wave stable and unstable regimes in the $(\beta,\e)$ plane. Since $G_U$ is strictly decreasing, we have $G_U \left(\frac{K}{\nu}\right)>\nu$ for $K < K_\ast$, and hence the flow is long-wave unstable, whereas $G_U \left(\frac{K}{\nu}\right) < \nu$ for $K > K_\ast$, and hence the flow is long-wave stable, as $\e\to0$, $\beta=K\e\to0$.

The condition
\begin{equation}\label{Yudovich-beta=Ke}
    G_U \left(\frac{K}{\nu}\right) = \left\| \left( \partial_{yy} - \frac{\im K}{\nu} \right)^{-1} U' \right\|_{L^2} > \nu
\end{equation}
extends Yudovich's condition \eqref{Yudovich}, which corresponds to the case $K=0$. The explicit threshold $K_\ast(U,\nu)$, depending on the shear profile $U$ and the viscosity $\nu$, shows that sufficiently strong rotation can suppress the instability mechanism of \cite{Colombo}, whereas weaker rotation allows it to persist. Thus, long-wave instability and long-wave stability coexist, separated by the threshold $K=K_\ast$.

(3) \emph{The unstable long waves form a band}.
Suppose that \eqref{Yudovich} holds and fix any $K_1 \in (0,K_\ast)$. In \eqref{eigenvalue-K} the leading-order term $G_U(\frac{K}{\nu})^2 - \nu^2$ is continuous and strictly decreasing in $K$, hence bounded below by $G_U(\frac{K_1}{\nu})^2-\nu^2>0$ on $[0,K_1]$, while the remainder is $\cO(\e)$ uniformly for $K \in [0,K_1]$. Therefore there is $\e_1=\e_1(U,\nu,K_1)>0$ such that the right-hand side of \eqref{eigenvalue-K} is positive for all $K\in[0,K_1]$ and all $0<\e<\e_1$; that is, the unstable region contains the truncated cone $\{0\le\beta\le K_1\e,\ 0<\e<\e_1\}$, and is contained in
$\{\beta<K_\ast\e\}$ up to the remainder. In this sense, the unstable region is fan-shaped in the $(\beta,\e)$ plane. Therefore, for sufficiently small $\beta \geq 0$ , there exist $\delta_L=\delta_L(U,\nu,\beta)$ and $\delta_R=\delta_R(U,\nu,\beta)$, such that
\begin{equation}\label{band}
  \frac{\beta}{K_\ast(U,\nu)} \;\leq\; \delta_L \; < \; \delta_R \;\leq\; \frac{\delta_U\,\nu}{1+\beta} ,
\end{equation}
and the flow is unstable for $\delta_L < \e < \delta_R$, while it is stable for sufficiently small $0 < \e < \delta_L$. Thus, the planetary rotation transforms the non-rotating picture in which ``all sufficiently long waves are unstable'' (\(\beta=0\), \cite{Colombo}) into one in which ``an intermediate band of long waves is unstable, while sufficiently long waves are stable''. 

\begin{Remark}
  The planetary rotation effect is reflected by $b=\tfrac{\beta}{\e\nu} = \frac{K}{\nu}$. The ratio $b$ carries a physical meaning: its numerator $K=\tfrac{\beta}{\e}$ is the frequency of the barotropic Rossby wave carried by the long-wave mode, and its denominator $\nu$ is the viscous decay rate of the $\cO(1)$ meridional modes with which that mode exchanges energy. The stabilization is thus of the same nature as the dispersive suppression of the lift-up effect for three-dimensional rotating Couette flow \cite{HuangW202409,HuangW202412,Zelati}, though it acts here on a two-dimensional, viscous, long-wave instability.
\end{Remark}

\end{Remark}


\subsection{Difficulties and novel ideas}\label{ssec:difficulties}


The overall structure of the proof follows [10], through deriving an exact resolvent
formula for the singular leading operator $\mathcal M_{\nu,\eps,\beta}$ and then
applying Kato's reduction [20] to transfer spectral information from
$\mathcal M_{\nu,\eps,\beta}$ to $\mathcal L_{\nu,\eps,\beta}$. The Coriolis $\beta$-term,
however, changes the analysis at every stage and requires new ingredients:
(i) it displaces the principal eigenvalue of $\mathcal M_{\nu,\eps,\beta}$ off the real
axis,
(ii) it makes every constant in the resolvent estimates, spectral projections and
Neumann inversions of [10] depend on $\beta$, and
(iii) it shifts the terms that determine stability to second order in $\eps$.
Throughout, the estimates are required to be uniform in $\beta \ge 0$; in
particular, no lower bound on $\beta$ is imposed at any point, so the rotating and
the non-rotating regimes are covered by a single argument.
This uniformity is based
on writing the resolvent of $\mathcal M_{\nu,\eps,\beta}$ in the form
\eqref{cMresolvent}, whose only singularity is the eigenvalue $\mu_0$ itself.
We highlight three main points.

\medskip
\noindent (1) \emph{The singular perturbation is confined to a rank-one direction.} 
The splitting \eqref{decomposition L} places the entire $\e^{-1}$ singularity in $\cM_{\nu,\e,\beta}$, a rank-one modification of the Fourier multiplier $\nu\cD_\e$.
Two exact algebraic facts make the scheme work: 
(i) Since $U$ has zero average, the zero-mode component of the perturbation gains
two powers of $\eps$, and the $\beta$-part drops out altogether; see Lemma \ref{lem:Cancelation} for details.
Consequently, this compensates the $\e^{-1}$ in $\cM_{\nu,\e,\beta}$.
(ii) The rank-one structure yields a closed
resolvent formula through the Sherman--Morrison identity; see Lemma \ref{lem:ResolventcM} for details.

\medskip
\noindent (2) \emph{Rotation moves the principal eigenvalue off the real axis, and the resolvent formula develops an apparent pole.} The Coriolis term shifts the principal eigenvalue of $\cM_{\nu,\e,\beta}$ to $\nu\mu_0=-\nu\e^2+\im\beta/\e$, that is $\mu_0=-\e^2+\im b$, while $\sigma_{L^2}(\cD_\e)$, against which its resolvent is compared, stays real. At $\beta=0$ the two spectra coincide. For $\beta\neq0$, the pole $\zeta= -\e^2$ is removable, leaving a representation \eqref{cMresolvent} analytic apart from the simple pole at $\mu_0$ (Remark \ref{rem:removable}). Crucially, in \eqref{cMresolvent} the resolvent of $\cD_\e$ occurs only as $(\cD_\e-\zeta)^{-1}\Pi_{\neq}$, so the zero mode of $\cD_\e$, and with it $-\e^2$, disappears from every subsequent estimate; see Lemma \ref{lem:Spectrum-M} for details. 

\medskip
\noindent (3) \emph{The eigenvalue expansion is governed by the order of the poles.} 
Rather than approximating the perturbed eigenfunction by the unperturbed one, as in \cite{Colombo}, we expand every component of $\cL_{\nu,\e,\beta}$ and of the eigenvector $V_{\nu,\e,\beta}$ to second order in $\e$. The two are given by the same contour integral over $\Gamma_0$, differing only by a factor $\nu\zeta$ in the integrand.
Consequently, if a term of the expansion has a simple pole at $\mu_0$ as its only singularity inside
$\Gamma_0$, its two contributions differ exactly by the factor $\nu\mu_0$.
Since the imaginary part does not affect the spectral stability, the leading order is $\nu \e^2$. One therefore knows in advance that only terms with a pole of order at least two can matter, namely those carrying the zero-mode factor, and that up to second order exactly one such term survives, which contributes the destabilizing term $\frac{\e^2}{\nu}G_U(b)^2$ of \eqref{eigenvalue-GU}, while $\Re(\nu\mu_0)$ supplies the viscous damping $-\nu\e^2$. This explains the cancellation of all first-order terms. See subsections 5.3-5.5 for detailed calculations.


\subsection{The structure of the paper}

The rest of the paper is organized as follows: In Section \ref{sec operators}, we analyze and present some basic properties of the key operators $\cM_{\nu, \e ,\beta}$ and $\cR_{\e,\beta}$ defined in \eqref{Operator-M}--\eqref{Operator-R}. 
Section \ref{sec spectral} is devoted to spectral projections and resolvent bounds for these operators, establishing estimates crucial to the perturbative analysis. In Section \ref{sec Kato}, we apply Kato’s reduction method to relate the spectrum and the eigenspace decomposition of the full operator $\cL_{\nu,\e,\beta}$ to those of the unperturbed operator $\cM_{\nu,\e,\beta}$, providing a framework to compute eigenvalue expansions. In Section \ref{sec expansion}, we derive the asymptotic expansion of the principal eigenvalue $\lambda_{\nu,\e,\beta}^{(0)}$ in the long-wave parameter $\e$, carefully tracking the contributions from the Rossby parameter $\beta$, and complete the proof of Theorem \ref{thm:main-1}.

\bigskip

\section{The Operators $\cM_{\nu,\e,\beta}$ and $\cR_{\e,\beta}$}\label{sec operators}

In this section, we analyze the leading and perturbative components of the operator $\cL_{\nu, \e ,\beta}$, as defined in \eqref{Operator-M} and \eqref{Operator-R}. Throughout, we follow the general approach of Section 2.1 in \cite{Colombo}, developed for the non-rotating case, and adapt it to account for the Rossby parameter $\beta$.

Throughout Sections \ref{sec operators}--\ref{sec expansion}, we fix $0<\nu\leq1$, $\e>0$ and $\beta \geq 0$, and let $b$ and $\theta_{\nu,\e,\beta}$ be defined as in \eqref{def-b}, \eqref{def-theta}, respectively.
Note that the restriction $\nu \leq 1$ is imposed only to simplify the constants in the
estimates below, and \eqref{keyassumption} forces $\eps < \nu$, so the regime
of interest is $\eps \ll \nu \leq 1$.
For simplicity, we suppress the subscripts and write $\theta_{\nu,\e,\beta}$ as $\theta$.
We also set
\begin{equation}\label{def-Theta}
  \Theta \coloneqq \theta + \e\nu ,
\end{equation}
which will turn out to be the natural size of the principal eigenfunction, and we record for later use the elementary inequality
\begin{equation}\label{Theta-vs-theta}
  \theta\,(1+\beta) \; \leq \; \Theta.
\end{equation}
Finally, all constants denoted by $C$, $C_U$, $\delta_U$ depend only on $\|U\|_{C^2}$ and never on $\e$, $\nu$, $\beta$ or $j$.

\medskip

We begin by establishing some useful cancellations.
\begin{Lemma}\label{lem:Cancelation}
Let $\cR_{\e,\beta}$ denote the operator defined in \eqref{Operator-R}. Then, the following estimate holds:
\begin{equation}\label{Bound-R}
  \|\cR_{\e,\beta}\|_{\cB(L^2,L^2)}\leq \|U\|_{C^2}+\beta.
\end{equation}
Moreover, for all $f \in L^2(\TT)$ we have
\begin{equation}\label{Cancelation}
    \Pi_0 \cR_{\e,\beta}  f = \e^2 \left \langle (\de_y+\e)^{-1} U,(\de_y+\e)^{-1} \Pi_{\neq} f \right\rangle_{L^2}.
\end{equation}
In particular, $\Pi_0 \cR_{\e,\beta} 1=0$ and
\begin{equation}\label{estimatePi0cR}
    \|\Pi_0 \cR_{\e,\beta}\|_{\cB(L^2,\CC)} = \e^2 \| (-\de_{yy}+\e^2)^{-1} U \|_{L^2} \leq \e^2 \|U\|_{L^2}.
\end{equation}
\end{Lemma}
\begin{proof}
By the definition of $\cR_{\e,\beta}$, for any $f \in L^2$, we have
\begin{align*}
  \|\cR_{\varepsilon,\beta} f\|_{L^2} & = \| U f + (U''-\beta)(-\partial_{yy}+\varepsilon^2)^{-1}\Pi_{\neq}f\|_{L^2} \\
  & \leq \|U\|_{\infty} \|f\|_{L^2} + \|U''\|_{\infty} \|(-\partial_{yy}+\varepsilon^2)^{-1}\Pi_{\neq}f\|_{L^2} + \beta \|(-\partial_{yy}+\varepsilon^2)^{-1}\Pi_{\neq}f\| _{L^2} \\
  & \leq (\|U\|_{C^2} + \beta) \|f\|_{L^2},
\end{align*}
which yields \eqref{Bound-R}. Here, we used the fact
$$
\|(-\partial_{yy}+\varepsilon^2)^{-1}\Pi_{\neq}f\|_{L^2} = \left( \sum_{\ell \neq 0} \frac{|f_{\ell}|^2}{(\ell^2 + \e^2)^2} \right)^{\frac{1}{2}} \leq \|f\|_{L^2}.
$$

Next, since $U$ has zero average on $\mathbb{T}$, for $f \in L^2$ we compute
\begin{align*}
\Pi_0\cR_{\e,\beta} f &= \langle U , \Pi_{\neq}f \rangle_{L^2} + \langle U''-\beta, (-\de_{yy}+\e^2)^{-1} \Pi_{\neq}f \rangle_{L^2}  \\
& = \langle U , \Pi_{\neq}f \rangle_{L^2} + \langle U'', (-\de_{yy}+\e^2)^{-1} \Pi_{\neq}f \rangle_{L^2} \\
& = \langle U , \Pi_{\neq}f \rangle_{L^2} + \langle U, \de_{yy}(-\de_{yy}+\e^2)^{-1} \Pi_{\neq}f \rangle_{L^2} \\
& = \langle U , \Pi_{\neq}f \rangle_{L^2} + \langle U, (\de_{yy} - \e^2 + \e^2)(-\de_{yy}+\e^2)^{-1} \Pi_{\neq}f \rangle_{L^2} \\
& = \e^2 \langle U, (-\de_{yy}+\e^2)^{-1} \Pi_{\neq}f \rangle_{L^2} \\
& = \e^2 \langle (-\de_{yy}+\e^2)^{-1} U, \Pi_{\neq} f \rangle_{L^2} \\
& = \e^2 \langle (\de_{y} + \e)^{-1} U, (\de_{y} + \e)^{-1} \Pi_{\neq} f \rangle_{L^2} ,
\end{align*}
which leads to \eqref{Cancelation}. Here, we have used the self-adjointness of $(-\de_{yy}+\e^2)^{-1}$ in the fifth line. For \eqref{estimatePi0cR}, viewing $\Pi_0\cR_{\e,\beta}$ as the linear functional
$$
f\mapsto \e^2\langle (-\de_{yy}+\e^2)^{-1}U, \Pi_{\neq}f\rangle_{L^2} = \e^2\langle (-\de_{yy}+\e^2)^{-1}U, f\rangle_{L^2}.
$$
The last equality holds since $U$ has no zero Fourier mode and $(-\de_{yy}+\e^2)^{-1}U$ then has zero average. Then the Riesz representation theorem gives
\[
\|\Pi_0 \cR_{\e,\beta}\|_{\cB(L^2,\CC)} = \sup_{\|f\|_{L^2}=1} \e^2 |\langle (-\de_{yy}+\e^2)^{-1}U, f\rangle_{L^2}| = \e^2 \|(-\de_{yy}+\e^2)^{-1}U\|_{L^2},
\]
with the extremizer 
$$f=\frac{(-\de_{yy}+\e^2)^{-1}U}{\|(-\de_{yy}+\e^2)^{-1}U\|_{L^2}}.$$
This proves the equality in \eqref{estimatePi0cR}. The inequality follows since $(-\de_{yy}+\e^2)^{-1}$ has operator norm at most $1$ on $\Pi_{\neq}$-modes (as $\ell^2+\e^2\geq 1$ for every $\ell \neq 0$).
\end{proof}
We now record basic spectral properties of $\mathcal{M}_{\nu, \varepsilon,\beta}$, in the same spirit as Lemma 2.3 in \cite{Colombo}. The main difference is that the presence of $\beta$ shifts the principal eigenvalue $\mu_0$ off the real axis. 

We first introduce the rotation-modified operator
\begin{equation}\label{def-Dnueb}
  \cD_{\nu,\e,\beta} \coloneqq \partial_{yy} - \frac{\im\beta}{\e\nu} = \partial_{yy} - \im b .
\end{equation}
For every $\beta \geq 0$ the restriction of $\cD_{\nu,\e,\beta}$ to the zero-average subspace $\Pi_{\neq}L^2(\TT)$ is boundedly invertible, since its eigenvalues there are $-j^2-\im b \neq 0$ for $j \neq 0$, and
\begin{equation}\label{bd-Dnueb}
  \big\| \cD_{\nu,\e,\beta}^{-m}\Pi_{\neq} \big\|_{\cB(L^2,L^2)} = \max_{j \neq 0} \frac{1}{(j^4+b^2)^{m/2}} = \frac{1}{(1+b^2)^{m/2}} \leq \theta^{m} , \qquad m \in \NN .
\end{equation}
Accordingly we set
\begin{equation}\label{exp-F}
  F_{\nu,\e,\beta} \coloneqq \cD_{\nu,\e,\beta}^{-1}\Pi_{\neq}(U'') , \qquad
  V_0 \coloneqq F_{\nu,\e,\beta} - \im\e\nu ,
\end{equation}
which is meaningful for all $\beta \geq 0$ because $U''$ has zero average. When $\beta>0$, we have
$$F_{\nu,\e,\beta}-\im\e\nu = \cD_{\nu,\e,\beta}^{-1}(U''-\beta),$$  and when $\beta=0$, we have $F_{\nu,\e,0}=U$ and $V_0 = U-\im\e\nu$, which is the principal eigenfunction of \cite[Lemma 2.3]{Colombo}.
\begin{Lemma}\label{lem:Spectrum-M}
Let $\e  \neq 0$, $\beta \geq 0$, and let $\cM_{\nu, \e,\beta}: \cD(\cM_{\nu, \e,\beta}) =  H^2(\TT) \subset L^2(\TT) \to L^2(\TT)$ be the operator defined in \eqref{Operator-M}. 
Then,
\begin{equation}\label{Spectrm-M}
 \sigma_{L^2}(\cM_{\nu, \e,\beta}) = \left\{ -\nu\e^2 +\frac{\im \beta}{\e} \right\} \cup \{-\nu(j^2+\e^2)\}_{j\in \mathbb{Z} \setminus \{0\}} ,
\end{equation}
the first eigenvalue being simple. Moreover, the set of eigenfunctions $\{V_0, \ e^{\im j y}\}_{j\in \ZZ \setminus \{0\} }$, with $V_0$ as in \eqref{exp-F} and associated with the eigenvalues $\{-\nu\e^2 +\frac{\im \beta}{\e}, \ -\nu(j^2+\e^2)\}_{j\in \mathbb{Z} \setminus \{0\}}$, is a Riesz basis of $L^2(\TT)$, being a rank-one modification of an orthonormal basis. 
\end{Lemma}
\begin{proof}
For $j \neq 0$, it is clear that
$$\cM_{\nu, \e,\beta}(e^{\im jy}) = \nu \cD_{\e}(e^{\im jy}) = -\nu(j^2+\e^2) e^{\im jy}.$$
 For $j=0$, write $F=F_{\nu,\e,\beta}$ and note that $\Pi_0 V_0 = -\im\e\nu$, since $F$ has zero average. Hence
\[
\cM_{\nu,\e,\beta}V_0 = \nu\cD_\e V_0 - \frac{\im}{\e}(U''-\beta)\,\Pi_0V_0 = \nu\cD_\e V_0 - \nu (U''-\beta) .
\]
By \eqref{def-Dnueb} and \eqref{exp-F} we have
$$\nu\partial_{yy}F = \nu\cD_{\nu,\e,\beta}F + \im\nu b F = \nu U'' + \frac{\im\beta}{\e}F,$$
 and therefore
\[
\nu\cD_\e V_0 = \nu\partial_{yy}F - \nu\e^2 F + \im\e^3\nu^2
= \nu U'' + \Big(\frac{\im\beta}{\e}-\nu\e^2\Big) F + \im\e^3\nu^2 .
\]
Combining the last two displays and using $\big(\frac{\im\beta}{\e}-\nu\e^2\big)(-\im\e\nu) = \im\e^3\nu^2 + \nu\beta$, we obtain
\[
\cM_{\nu,\e,\beta}V_0 = \Big(-\nu\e^2 + \frac{\im\beta}{\e}\Big)\big(F - \im\e\nu\big) = \Big(-\nu\e^2 + \frac{\im\beta}{\e}\Big) V_0 .
\]
Since $\Pi_0V_0=-\im\e\nu \neq 0$, the function $V_0$ is linearly independent of $\{e^{\im jy}\}_{j\neq0}$, so $\{V_0,e^{\im jy}\}_{j\neq0}$ is a basis of $L^2(\TT)$. In particular the eigenvalue $-\nu\e^2+\frac{\im\beta}{\e}$ is simple.
\end{proof}

For simplicity of the calculation later, we denote the eigenvalues of $\frac{1}{\nu}\cM_{\nu, \e,\beta}$ by $\mu_j$, namely
\[
\mu_j =
\begin{cases}
  -\e^2 + \frac{\im \beta}{\e \nu}, & \mbox{if $j =0$} \\
  -(j^2+\e^2), & \mbox{if $j \neq 0$}.
\end{cases}
\]

Note that $\mu_0$ is an isolated simple eigenvalue, while $\mu_j$, $j\neq 0$, are isolated double eigenvalues. In computing the resolvent of $\cM_{\nu,\e,\beta}$, we notice that since $\cM_{\nu,\e,\beta}$ is a rank-$1$ modification of the operator $\nu \cD_{\e}$, the resolvent of $\cM_{\nu,\e,\beta}$ has a more explicit form. In order to show this, we refer to the appendix of \cite{Colombo} or \cite{Deng} for the following Sherman-Morrison formula.

\begin{Lemma}\label{lem:ShermanMorrisoninf}
Let $H$ be a complex Hilbert space whose inner product $\langle\cdot,\cdot\rangle_H$ is, as in \eqref{innerproduct}, conjugate-linear in the first and linear in the second variable, so that $f\langle g,\cdot\rangle_H$ is a bounded linear rank-one operator. Let $\mathcal{A}:D(\mathcal{A})\subset H \longrightarrow H $ be a closed, densely-defined, linear operator with a bounded inverse $\mathcal{A}^{-1}$, and let $f,g \in H$.
Then, the operator $\mathcal{A} + f\langle g, \cdot \rangle_H$ is invertible if and only if $\langle g, \mathcal{A}^{-1} f \rangle_H+1\neq 0$. Moreover,
\begin{equation}\label{originalShermanMorrison}
\left( \mathcal{A}+f\langle g, \cdot \rangle_H \right)^{-1}=\mathcal{A}^{-1}-\frac{\mathcal{A}^{-1}(f\langle g, \cdot \rangle_H)\mathcal{A}^{-1}}{1+\langle g, \mathcal{A}^{-1} f\rangle_H}.
\end{equation}
\end{Lemma}

Applying the preceding lemma, we obtain the resolvent of $\cM_{\nu,\e,\beta}$. It is important for what follows to write it in a form in which the only singularity is the eigenvalue $\mu_0$ itself.
\begin{Lemma}\label{lem:ResolventcM}
Let $\beta \geq 0$, $\nu>0$, $0<\e<1$, and let
\[
  \zeta \in \CC \setminus \Big( \{\mu_0\} \cup \{-(\ell^2+\e^2)\}_{\ell \neq 0} \Big) .
\]
Then $\cM_{\nu,\e,\beta}-\nu\zeta$ is boundedly invertible and
\begin{align}\label{cMresolvent}
        (\cM_{\nu,\e,\beta} -\nu\zeta)^{-1} = \frac{1}{\nu} (\cD_\e -\zeta)^{-1}\Pi_{\neq}  +  \frac{h_{\nu,\e}(\zeta,y)}{\nu (\zeta - \mu_0)}\Pi_0,
  \end{align}
where
\begin{equation}\label{def-h}
  h_{\nu,\e}(\zeta,y) \coloneqq \frac{u_{\e}(\zeta,y)}{\im\e\nu} - 1 ,
  \qquad
  u_{\e}(\zeta,y) \coloneqq (\cD_{\e} -\zeta)^{-1} (U'') = \sum_{j \neq 0} \frac{j^2 U_j\, e^{\im j y}}{j^2+\e^2+\zeta} .
\end{equation}
In particular $\Pi_0 u_\e = 0$ and $\Pi_0 h_{\nu,\e} = -1$. The mappings $\zeta \mapsto u_\e(\zeta,\cdot)$, $\zeta \mapsto h_{\nu,\e}(\zeta,\cdot)$ and $\zeta \mapsto (\cD_\e-\zeta)^{-1}\Pi_{\neq}$ are analytic on $\CC\setminus\{-(\ell^2+\e^2)\}_{\ell\neq0}$, hence in particular on the half-plane $\{\Re\,z>-1\}$, and consequently
\begin{equation}\label{resolvent-analytic}
  \zeta \longmapsto (\cM_{\nu,\e,\beta}-\nu\zeta)^{-1}
  \quad\text{is analytic on } \CC\setminus\big(\{\mu_0\}\cup\{-(\ell^2+\e^2)\}_{\ell\neq0}\big) ,
\end{equation}
with a simple pole at $\mu_0$. Note that the point $-\e^2$, which is an element of $\sigma_{L^2}(\cD_\e)$, does not appear in \eqref{resolvent-analytic}: it is a removable singularity of the right-hand side of \eqref{cMresolvent} whenever $\beta \neq 0$.
Finally,
\begin{equation}\label{h-at-mu0}
  h_{\nu,\e}(\mu_0,y) = \frac{V_0(y)}{\im\e\nu} ,
  \qquad
  \partial_\zeta h_{\nu,\e}(\mu_0,y) = \frac{\cD_{\nu,\e,\beta}^{-2}\Pi_{\neq}(U'')}{\im\e\nu} .
\end{equation}
\end{Lemma}
\begin{proof}
Assume first $\zeta \neq -\e^2$, and apply Lemma \ref{lem:ShermanMorrisoninf} with
$$H=L^2,\quad \cA=\nu(\cD_{\e}-\zeta),\quad f=-\frac{\im}{\e} (U''(y)-\beta), \quad\text{and}\quad g=1.$$
Here $\cA$ is invertible because, by hypothesis, $\zeta \notin \sigma_{L^2}(\cD_\e)=\{-(j^2+\e^2)\}_{j\in\ZZ}$. Since $\Pi_0(\cD_\e-\zeta)^{-1}(U'')=0$ and $(\cD_\e-\zeta)^{-1}\Pi_0 = -(\e^2+\zeta)^{-1}\Pi_0$, we get
  \begin{equation}\label{inv}
  \langle 1, \cA^{-1}(f) \rangle_{L^2} = \Pi_0 \left( -\frac{\im}{\e \nu}(\cD_{\e}-\zeta)^{-1}(U''(y)-\beta) \right)= -\frac{\im \beta}{\e \nu (\e^2 +\zeta)} ,
  \end{equation}
which equals $-1$ if and only if $\zeta = -\e^2+\im b = \mu_0$, excluded by hypothesis. Hence $\cM_{\nu,\e,\beta}-\nu\zeta$ is invertible and \eqref{originalShermanMorrison} gives
\begin{equation}\label{SMraw}
  (\cM_{\nu,\e,\beta} -\nu\zeta)^{-1} = \frac{1}{\nu} (\cD_\e -\zeta)^{-1}  +  \frac{u_{\e,\beta}(\zeta,y)}{\im \e \nu^2 (\zeta - \mu_0)}\Pi_0,
  \qquad u_{\e,\beta}(\zeta,y) \coloneqq u_\e(\zeta,y) + \frac{\beta}{\e^2+\zeta} .
\end{equation}
Both terms on the right-hand side of \eqref{SMraw} are singular at $\zeta=-\e^2$, but the singularities cancel. Indeed, splitting
$$(\cD_\e-\zeta)^{-1} = (\cD_\e-\zeta)^{-1}\Pi_{\neq} - (\e^2+\zeta)^{-1}\Pi_0$$ and using $\frac{\beta}{\im\e\nu} = -\im b$ together with $\zeta-\mu_0 = (\zeta+\e^2)-\im b$, the two terms carrying the factor $(\e^2+\zeta)^{-1}$ combine to
\[
\frac{\Pi_0}{\nu(\e^2+\zeta)}\left[ -1 - \frac{\im b}{\zeta-\mu_0} \right]
= \frac{\Pi_0}{\nu(\e^2+\zeta)} \cdot \frac{-(\zeta+\e^2)}{\zeta-\mu_0}
= -\frac{\Pi_0}{\nu(\zeta-\mu_0)} ,
\]
which is regular at $\zeta=-\e^2$. Substituting this into \eqref{SMraw} and recalling \eqref{def-h} yields \eqref{cMresolvent} for $\zeta \neq -\e^2$. Since the right-hand side of \eqref{cMresolvent} is analytic at $\zeta=-\e^2$ (note $\mu_0=-\e^2$ only when $\beta=0$, a case excluded by $\zeta \neq \mu_0$), the identity extends to $\zeta=-\e^2$ by continuity, and there $\cM_{\nu,\e,\beta}-\nu\zeta$ is invertible because the right-hand side of \eqref{cMresolvent} is a bounded two-sided inverse. This proves \eqref{resolvent-analytic}. The Fourier expansion in \eqref{def-h} is immediate from $(U'')_j = -j^2U_j$, and its poles are at $\zeta=-(j^2+\e^2)\leq -1-\e^2<-1$. Finally $\cD_\e-\mu_0 = \partial_{yy}-\im b = \cD_{\nu,\e,\beta}$, so $u_\e(\mu_0,\cdot)=F_{\nu,\e,\beta}$ and $\partial_\zeta u_\e(\mu_0,\cdot)=\cD_{\nu,\e,\beta}^{-2}\Pi_{\neq}(U'')$, which gives \eqref{h-at-mu0}.
\end{proof}
\begin{Remark}
  In view of \eqref{cMresolvent}, we observe that the resolvent of $\cM_{\nu,\e,\beta}$ is compact since it is the sum of a rank-one operator and a Fourier multiplier with symbol tending to $0$. Therefore, the spectrum of $\cM_{\nu,\e,\beta}$ consists of eigenvalues given in \eqref{Spectrm-M}.
\end{Remark}
\begin{Remark}\label{rem:removable}
Formula \eqref{cMresolvent} is the form of the resolvent that we shall use throughout. Its advantage over the raw Sherman--Morrison expression \eqref{SMraw} is that the point $\zeta=-\e^2$ plays no role in it. This matters because $-\e^2 \in \sigma_{L^2}(\cD_\e)$ whereas $-\e^2 \notin \sigma_{L^2}(\tfrac1\nu\cM_{\nu,\e,\beta})$ as soon as $\beta \neq 0$, so a contour encircling $\mu_0$ may pass on either side of $-\e^2$ depending on the size of $b$. Actually, using \eqref{SMraw}, one would have to know which, and each individual term would carry residues at $-\e^2$ proportional to negative powers of $\beta$. Using \eqref{cMresolvent}, the question does not arise. When $\beta=0$ we have $\mu_0=-\e^2$, $h_{\nu,\e}=\frac{u_\e}{\im\e\nu}-1$, and \eqref{cMresolvent} reduces to \cite[Lemma 2.5]{Colombo}.
\end{Remark}

\bigskip

\section{Spectral Projections and Resolvent Bounds}\label{sec spectral}

The idea of the Kato reduction method is to reduce a complicated perturbed operator problem to a simpler effective one, typically associated with an isolated part of the spectrum. In our setting, the goal is to transfer spectral information from the unperturbed operator $\cM_{\nu,\e,\beta}$ to the full operator $\cL_{\nu,\e,\beta}$ by constructing an isomorphism between the corresponding eigenspaces of $\cM_{\nu,\e,\beta}$ and $\cL_{\nu,\e,\beta}$. To this end, we first define the Riesz spectral projections associated with $\cM_{\nu,\e,\beta}$ and $\cL_{\nu,\e,\beta}$. Formally, these are given by
\[
Q_{\nu,\e,\beta} \coloneqq -\frac{1}{2\pi\im} \oint_{\Gamma} (\cM_{\nu,\e,\beta} -\lambda)^{-1} \drm \lambda,
\qquad
P_{\nu,\e,\beta} \coloneqq -\frac{1}{2\pi\im} \oint_{\Gamma} (\cL_{\nu,\e,\beta} -\lambda)^{-1} \drm \lambda,
\]
where $\Gamma \subset \rho_{L^2}(\cM_{\nu,\e,\beta}) \cap \rho_{L^2}(\cL_{\nu,\e,\beta})$ is a closed contour. By Lemma~\ref{lem:Spectrum-M}, the eigenvalues of $\cM_{\nu,\e,\beta}$ are $\nu\mu_j$, $j \in \NN$, and we take $\Gamma=\nu\Gamma_j$ with $\Gamma_j$ a small circle around $\mu_j$. For $j\geq1$, the eigenvalue $\mu_j=-(j^2+\e^2)$ is real. For $j=0$, the eigenvalue $\mu_0=-\e^2+\im b$ sits at height $b$ above the real axis. We therefore set
\begin{equation}\label{def-radius}
   \Gamma_j : \TT \ni \phi \longmapsto \mu_j + \tfrac12 e^{\im\phi} \in \CC ,
   \qquad j \in \NN ,
\end{equation}
that is, the circle of radius $\tfrac12$ centred at $\mu_j$, exactly as in \cite{Colombo}. 
\begin{Lemma}\label{lem:resolventD}
Let $\beta \geq 0$ and $0<\e<1$, and let $\Gamma_j$ be as in \eqref{def-radius}. Set
\begin{equation}\label{def-thetaj}
  \theta_j \coloneqq \begin{cases} \theta, & j=0, \\ 1, & j \geq 1 .\end{cases}
\end{equation}
Then the closed disks $\overline{D(\mu_j,\tfrac12)}$, $j\in\NN$, are pairwise disjoint, each $\Gamma_j$ encircles $\mu_j$ and no other point of $\sigma_{L^2}(\tfrac1\nu\cM_{\nu,\e,\beta})$, and
\begin{equation}\label{dist-Gamma}
  \mathrm{dist}\Big(\Gamma_0, \big\{-(\ell^2+\e^2)\big\}_{\ell \neq 0}\Big) \geq \tfrac12\max\{1,b\} = \tfrac12\,\theta^{-1},
  \quad
  \mathrm{dist}\Big(\Gamma_j, \big\{-(\ell^2+\e^2)\big\}_{\ell \neq 0}\Big) = \tfrac12 \quad (j\geq1) .
\end{equation}
Consequently
\begin{equation}\label{estimate-D}
 \big\| (\cD_\e - \zeta)^{-1}\Pi_{\neq} \big\|_{\cB(L^2,L^2)} \leq 2\,\theta_j \leq 2 , \qquad \forall \zeta \in \Gamma_j , \ j \in \NN ,
\end{equation}
and, in addition, $|\zeta-\mu_0| \geq \tfrac12$, hence $|\zeta-\mu_0|^{-1}\leq 2$, for every $\zeta\in\Gamma_j$ and every $j\in\NN$.
\end{Lemma}
\begin{proof}
For $\zeta = \mu_0 + \tfrac12e^{\im\phi}$ and $\ell \neq 0$,
\[
\big| -(\ell^2+\e^2) - \zeta \big| = \big| \ell^2 + \im b + \tfrac12 e^{\im\phi} \big| \geq \max\big\{\, \ell^2 - \tfrac12 ,\ b-\tfrac12 \,\big\} \geq \max\big\{\tfrac12,\ b-\tfrac12\big\} ,
\]
the two lower bounds coming from the real and the imaginary part respectively and the last step from $\ell^2\geq1$. 
This is the first part of \eqref{dist-Gamma}.

For $j\geq1$ the circle $\Gamma_j$ is centred at the real point $\mu_j$, and since $|\ell^2-j^2|\geq1$ for $\ell\neq\pm j$ we get
$$|-(\ell^2+\e^2)-\zeta| = |(j^2-\ell^2)-\tfrac12e^{\im\phi}| \geq \tfrac12$$ for every $\ell\neq0$, with equality for $\ell=\pm j$. No gain in $\theta$ is available for $j\geq1$, and none is needed. Since $(\cD_\e-\zeta)^{-1}\Pi_{\neq}$ is the Fourier multiplier with symbol $-(\ell^2+\e^2+\zeta)^{-1}$ on the modes $\ell\neq0$, \eqref{estimate-D} follows from \eqref{dist-Gamma}.

As for $|\zeta-\mu_0|$, it equals $\tfrac12$ on $\Gamma_0$, while for $j\geq1$ one has
$$|\zeta-\mu_0|\geq|\Re(\zeta-\mu_0)| \geq j^2-\tfrac12\geq\tfrac12.$$
 Finally $|\mu_j - \mu_\ell| \geq |\Re(\mu_j-\mu_\ell)| = |j^2-\ell^2| \geq 1$ for $j\neq\ell$, so the disks $\overline{D(\mu_j,\tfrac12)}$ are pairwise disjoint and each $\Gamma_j$ encircles $\mu_j$ alone.
\end{proof}
\begin{Remark}\label{rem:contour}
For $\beta=0$, we have $\theta=1$ and \eqref{estimate-D} is exactly the bound $\|(\cD_\e-\zeta)^{-1}\Pi_{\neq}\|\leq2$ of \cite[(2.12)]{Colombo}. For $\beta \gg \e\nu$, that is $b \gg 1$, the estimate \eqref{estimate-D} is stronger than its non-rotating counterpart by the factor $\theta$.

\end{Remark}

After the change of variables $\lambda = \nu \zeta$, we may introduce the auxiliary Riesz projection associated with $\cM_{\nu,\e,\beta}$:
\begin{equation}\label{auxiliaryRieszproj}
    Q_{\nu,\e,\beta}^{(j)} \coloneqq -\frac{\nu}{2\pi\im} \oint_{\Gamma_j} (\cM_{\nu,\e,\beta} -\nu \zeta)^{-1} \drm \zeta \;:\; L^2(\TT) \to L^2(\TT).
\end{equation}
The standard properties of the Riesz projection yield $[Q_{\nu,\e,\beta}^{(j)}]^2 = Q_{\nu,\e,\beta}^{(j)}$, and the following decomposition
\begin{equation}
L^2(\TT) = \mathrm{Ran}\, Q_{\nu,\e,\beta}^{(j)} \oplus \mathrm{Ker}\, Q_{\nu,\e,\beta}^{(j)},
\end{equation}
where $\mathrm{Ran}\, Q_{\nu,\e,\beta}^{(j)}$ is the eigenspace of $\cM_{\nu,\e,\beta}$ associated with the eigenvalue $\nu \mu_j$, and $\mathrm{Ker}\, Q_{\nu,\e,\beta}^{(j)}$ is the direct sum of all the remaining eigenspaces of $\cM_{\nu,\e,\beta}$.

Similarly, we expect to define the Riesz projection for $\cL_{\nu,\e,\beta}$ along the same contours
\begin{equation}\label{Rieszproj}
    P_{\nu,\e,\beta}^{(j)} \coloneqq -\frac{\nu}{2\pi\im} \oint_{\Gamma_j} (\cL_{\nu,\e,\beta} -\nu\zeta)^{-1} \drm \zeta \;:\; L^2(\TT) \to L^2(\TT).
\end{equation}
However, it is not immediate that $\cL_{\nu,\e,\beta} - \nu\zeta$ is invertible for $\zeta \in \Gamma_j$. To justify that \eqref{Rieszproj} is well defined, we need to establish several technical lemmas.

Observe that for every $\zeta \in \Gamma_j$, we have $\nu \zeta \in \rho_{L^2}(\cM_{\nu,\e,\beta})$, and hence, by Lemma \ref{lem:ResolventcM}, the operator $\cM_{\nu,\e,\beta} - \nu \zeta$ is invertible with bounded inverse. It follows that $Q_{\nu,\e,\beta}^{(j)}$ is a well-defined projection. We can isolate the resolvent of $\cM_{\nu,\e,\beta}$ from the resolvent of $\cL_{\nu,\e,\beta}$ to write
\begin{equation}\label{resolventcL1}
  (\cL_{\nu,\e,\beta} -\nu\zeta)^{-1}  = \big( \uno - \im \e (\cM_{\nu,\e,\beta} - \nu \zeta)^{-1}  \cR_{\e,\beta} \big)^{-1} (\cM_{\nu,\e,\beta} -\nu\zeta)^{-1}.
\end{equation}
By a Neumann series expansion, it follows that the resolvent of $\cL_{\nu,\e,\beta}$ can be written as
\begin{equation}\label{resolventcL2}
  (\cL_{\nu,\e,\beta} -\nu\zeta)^{-1} =  \sum_{p=0}^{\infty} \big( \im \e  (\cM_{\nu,\e,\beta} -\nu\zeta)^{-1} \cR_{\e,\beta}\big)^{p}  (\cM_{\nu,\e,\beta} -\nu\zeta)^{-1},
\end{equation}
provided that the operator $\uno - \im \e (\cM_{\nu,\e,\beta} - \nu \zeta)^{-1}\cR_{\e,\beta}$ is invertible and the corresponding Neumann series converges. The following lemma justifies the arguments.
\begin{Lemma}[Resolvent bounds for $\cL_{\nu,\e,\beta} $]\label{lem:inverseMR}
Let $j\in \NN$, $ \zeta \in \Gamma_j$, $\beta \geq 0$, $0<\nu\leq1$ and $0 <\e <1$ be such that
\begin{equation}\label{smallcond}
\frac{\e}{\nu} \big(1+\|U\|_{C^2}\big)^2(1+\beta) < \frac{1}{20}.
\end{equation}
Then,
\begin{equation}\label{condition-uniform}
\| \im \e (\cM_{\nu,\e,\beta} -\nu\zeta)^{-1} \cR_{\e,\beta}\|_{\cB(L^2,L^2)} \leq \frac{8\,\e\,\theta_j}{\nu}\big(1+\|U\|_{C^2}\big)^2(1+\beta) \; < \; \frac12.
\end{equation}
In particular, $\cL_{\nu,\e,\beta} -\nu\zeta$ is invertible with bounded inverse
\begin{equation}\label{estimate-inverseresovlentL}
    \|(\cL_{\nu,\e,\beta} -\nu\zeta)^{-1}\|_{\cB(L^2,L^2)} < 2 \|(\cM_{\nu,\e,\beta} -\nu\zeta)^{-1}\|_{\cB(L^2,L^2)} < \infty.
\end{equation}
\end{Lemma}
\begin{proof}
By Lemma \ref{lem:ResolventcM} and \eqref{cMresolvent},
\begin{equation}\label{auxiteration1}
        \im \e (\cM_{\nu,\e,\beta} -\nu\zeta)^{-1} \cR_{\e,\beta} = \underbrace{\frac{\im \e}{\nu} (\cD_\e -\zeta)^{-1}\Pi_{\neq} \cR_{\e,\beta}}_{=:\,\cT_1(\zeta)}  +  \underbrace{\frac{\im\e\, h_{\nu,\e}(\zeta,y)}{\nu(\zeta - \mu_{0})}\Pi_0 \cR_{\e,\beta}}_{=:\,\cT_2(\zeta)} .
    \end{equation}
For the first term, by \eqref{Bound-R} and \eqref{estimate-D},
\begin{equation}\label{lem:inverseMR-term-1}
    \|\cT_1(\zeta)\|_{\cB(L^2,L^2)} \leq \frac{2\e\theta_j}{\nu}(\|U\|_{C^2} + \beta).
\end{equation}
For the second term, note that $u_\e(\zeta,\cdot)=(\cD_\e-\zeta)^{-1}\Pi_{\neq}(U'')$ because $U''$ has zero average, so \eqref{def-h} and \eqref{estimate-D} give $\|u_\e(\zeta,\cdot)\|_{L^2}\leq 2\theta_j\|U''\|_{L^2}$ and hence
\[
\|h_{\nu,\e}(\zeta,\cdot)\|_{L^2} \leq \frac{2\theta_j\|U''\|_{L^2}}{\e\nu} + 1 ,
\]
while $\| \Pi_0 \cR_{\e,\beta}\|_{\cB(L^2,\CC)} \leq \e^2 \|U\|_{L^2}$ by \eqref{estimatePi0cR}. Since a rank-one operator has operator norm equal to the product of the norms of its two factors, and $|\zeta-\mu_0|^{-1}\leq 2$ by Lemma \ref{lem:resolventD},
\begin{equation}\label{lem:inverseMR-term-2}
  \|\cT_2(\zeta)\|_{\cB(L^2,L^2)} \leq \frac{2\e^3\|U\|_{L^2}}{\nu}\left( \frac{2\theta_j\|U''\|_{L^2}}{\e\nu} + 1\right)
  = \frac{4\,\e^2\theta_j}{\nu^2}\|U\|_{L^2}\|U''\|_{L^2} + \frac{2\e^3}{\nu}\|U\|_{L^2} .
\end{equation}
We compare the two contributions in \eqref{lem:inverseMR-term-2} with \eqref{lem:inverseMR-term-1}.
First, $$\frac{\e^2\theta_j}{\nu^2} = \frac{\e\theta_j}{\nu}\cdot\frac{\e}{\nu} \leq \frac{\e\theta_j}{\nu}.$$
 Second, we claim that
\begin{equation}\label{eps2-vs-theta}
  \e^2 \leq \theta_j(1+\beta) .
\end{equation}
Indeed, this is clear if $\theta_j=1$, since $\e<1$, and if $\theta_j=\theta=\frac{\e\nu}{\beta}<1$ then $\theta(1+\beta)\geq\theta\beta = \e\nu \geq \e^2$, because $\e\leq\nu$ by \eqref{smallcond}. Hence
$$\frac{\e^3}{\nu}\leq\frac{\e\theta_j(1+\beta)}{\nu}.$$
 Substituting these two observations into \eqref{lem:inverseMR-term-2}, adding \eqref{lem:inverseMR-term-1}, and bounding $\|U\|_{L^2}\leq\|U''\|_{L^2}\leq\|U\|_{C^2}$ together with $\|U\|_{C^2}+\beta\leq(1+\|U\|_{C^2})(1+\beta)$, we obtain
\begin{align*}
  \|\cT_1(\zeta)\|+\|\cT_2(\zeta)\| & \leq \frac{\e\theta_j}{\nu}\Big( 2(1+\|U\|_{C^2})(1+\beta) + 4\|U\|_{C^2}^2 + 2(1+\beta)\|U\|_{C^2}\Big) \\
  & \leq \frac{8\,\e\,\theta_j}{\nu}(1+\|U\|_{C^2})^2(1+\beta),
\end{align*}
which is the first inequality in \eqref{condition-uniform}. The second follows from \eqref{smallcond}. Then by the Neumann expansion \eqref{resolventcL2},
\begin{equation}\label{estimate-Neumanninverse}
  \|(  \uno-\im \e (\cM_{\nu,\e,\beta} -\nu\zeta)^{-1} \cR_{\e,\beta})^{-1} \|_{\cB(L^2,L^2)} \leq \frac{1}{1-\frac12} = 2 .
\end{equation}
By \eqref{resolventcL1}, \eqref{estimate-Neumanninverse} and Lemma \ref{lem:ResolventcM}, we obtain \eqref{estimate-inverseresovlentL}, and thus $\cL_{\nu,\e,\beta} -\nu\zeta$ is invertible with bounded inverse.
\end{proof}
We record separately the exact cancellation on the zero mode, which is the algebraic counterpart of Remark \ref{rem:removable} and which we shall use repeatedly. Since $\Pi_0(\cD_\e-\zeta)^{-1}\Pi_{\neq}=0$ and $\Pi_0h_{\nu,\e}=-1$, formula \eqref{auxiteration1} gives at once
\begin{equation}\label{Pi0-cancellation}
   \Pi_0 \big( \im \e(\cM_{\nu,\e,\beta} -\nu\zeta)^{-1}  \cR_{\e,\beta} \big) =  -\frac{\im \e}{\nu(\zeta - \mu_0)} \Pi_0 \cR_{\e,\beta} ,
\end{equation}
with no residual singularity at $\zeta=-\e^2$.

Under assumption \eqref{smallcond}, the Riesz projection associated with $\cL_{\nu,\e,\beta}$ is well defined.

Next, we present the resolvent bounds outside the contours. Set
\begin{equation}\label{def-Omega}
  \Omega_{\nu,\e,\beta} \coloneqq \CC \setminus \bigcup_{j \in \NN} \overline{D\big(\mu_j,\tfrac12\big)} ,
  \qquad D(\mu,r)\coloneqq\{z: |z-\mu|<r\} ,
\end{equation}
where $\mu_j$ is given in \eqref{def-radius}.
\begin{Lemma}[Resolvent bounds outside the contours]\label{lem:resolventOutside}
Assume \eqref{smallcond}, and let $\zeta \in \Omega_{\nu,\e,\beta}$. Then
\begin{equation}\label{outside-D}
  \big\|(\cD_\e-\zeta)^{-1}\Pi_{\neq}\big\|_{\cB(L^2,L^2)} \leq 2 ,
  \qquad
  |\zeta-\mu_0|^{-1} \leq 2 ,
  \qquad
  \|u_\e(\zeta,\cdot)\|_{L^2} \leq 2\|U''\|_{L^2} ,
\end{equation}
the operator $\cM_{\nu,\e,\beta}-\nu\zeta$ is boundedly invertible, and
\begin{equation}\label{outside-neumann}
  \big\| \im\e\,(\cM_{\nu,\e,\beta}-\nu\zeta)^{-1}\cR_{\e,\beta} \big\|_{\cB(L^2,L^2)}
  \leq \frac{8\,\e}{\nu}\big(1+\|U\|_{C^2}\big)^2(1+\beta) < \frac{2}{5} .
\end{equation}
Consequently $\cL_{\nu,\e,\beta}-\nu\zeta$ is boundedly invertible for every $\zeta \in \Omega_{\nu,\e,\beta}$.
\end{Lemma}
\begin{proof}
For $\ell\neq0$ we have $-(\ell^2+\e^2)=\mu_{|\ell|}$, so $\zeta \notin \overline{D(\mu_{|\ell|},\tfrac12)}$ gives $|\zeta+\ell^2+\e^2|>\tfrac12$. Since $(\cD_\e-\zeta)^{-1}\Pi_{\neq}$ is the Fourier multiplier with symbol $-(\ell^2+\e^2+\zeta)^{-1}$ on the modes $\ell\neq0$, the first bound in \eqref{outside-D} follows, and the third follows from it since $u_\e(\zeta,\cdot)=(\cD_\e-\zeta)^{-1}\Pi_{\neq}(U'')$. The second is $|\zeta-\mu_0|>\tfrac12$. Thus $\zeta$ satisfies on $\Omega_{\nu,\e,\beta}$ the very same two bounds that Lemma \ref{lem:resolventD} provides on the contours, with $\theta_j$ replaced by $1$. In particular, $\zeta$ lies in the domain of validity of Lemma \ref{lem:ResolventcM}, so $\cM_{\nu,\e,\beta}-\nu\zeta$ is invertible, \eqref{cMresolvent} holds, and the proof of Lemma \ref{lem:inverseMR} applies verbatim and gives \eqref{outside-neumann}. The Neumann series \eqref{resolventcL2} then converges, and \eqref{resolventcL1} shows that $\cL_{\nu,\e,\beta}-\nu\zeta$ is boundedly invertible.
\end{proof}

Our next goal is to construct an isomorphism between the ranges of $P_{\nu,\e,\beta}^{(j)}$ and $Q_{\nu,\e,\beta}^{(j)}$. In classical perturbation theory, this is typically achieved by assuming that $\|(P-Q)^2\|<1$, or more generally that the spectral radius of $(P-Q)^2$ is smaller than 1. However, as observed in \cite{Colombo}, the difference $P_{\nu,\e,\beta}^{(j)}-Q_{\nu,\e,\beta}^{(j)}$ is close to a rank-one projection, whose norm is generally not small. Consequently, the standard perturbative argument cannot be applied directly. Instead, we establish the following decomposition, which is analogous to Lemmas~2.7 and~2.9 in \cite{Colombo}. As in \cite{Colombo}, the zeroth Fourier mode of the ``large'' component $W_{\nu,\e,\beta}^{(j)}$ need not vanish. Thanks to \eqref{Pi0-cancellation} it is nevertheless of size $\cO(\e^2)$, uniformly in $\beta \geq 0$.
\begin{Lemma}\label{lem:P-Q}
Assume \eqref{smallcond}. Then there exists a constant $C_{U} >0$, depending on $\|U \|_{C^2}$ but independent of $\e$, $\nu$, $j$ and $\beta$, such that for every $j\in\NN$ the Riesz projections $Q_{\nu,\e,\beta}^{(j)}$ and $ P_{\nu,\e,\beta}^{(j)}$ satisfy
\begin{equation}\label{cT-decomp}
    P_{\nu,\e,\beta}^{(j)} - Q_{\nu,\e,\beta}^{(j)} = \cK_{\nu,\e,\beta}^{(j)}  + W_{\nu,\e,\beta}^{(j)} \Pi_0,
\end{equation}
where
\begin{align}
  \cK_{\nu,\e,\beta}^{(j)} &\coloneqq -\frac{1}{2\pi\im} \oint_{\Gamma_j} \cB(\zeta)\,\cA(\zeta)\,(\cD_\e-\zeta)^{-1}\Pi_{\neq} \,\drm\zeta , \label{def-cTj}\\
  W_{\nu,\e,\beta}^{(j)} &\coloneqq -\frac{1}{2\pi\im} \oint_{\Gamma_j} \frac{\cB(\zeta)\,\cA(\zeta)\big[h_{\nu,\e}(\zeta,\cdot)\big]}{\zeta-\mu_0} \,\drm\zeta . \label{def-Wj}
\end{align}
with
\[
\cA(\zeta) \coloneqq \im\e\,(\cM_{\nu,\e,\beta}-\nu\zeta)^{-1}\cR_{\e,\beta},
\qquad
\cB(\zeta) \coloneqq \big(\uno-\cA(\zeta)\big)^{-1}.
\]
Moreover, it holds that
\begin{align}
\label{bd:Pi0F1F2}
      &\frac{\nu}{\e\,(1+\beta)}\|\cK_{\nu,\e,\beta}^{(j)}\|_{\cB(L^2,L^2)} + \frac{\nu}{\e^3}\|\Pi_0 \cK_{\nu,\e,\beta}^{(j)}\|_{\cB(L^2,\CC)}\nonumber\\
      &\quad + \frac{\nu^2}{1+\beta}\|W_{\nu,\e,\beta}^{(j)}\|_{L^2} + \frac{\nu^2}{\e^2}|\Pi_0 W_{\nu,\e,\beta}^{(j)}| \leq C_{U}.
\end{align}
\end{Lemma}
\begin{proof}
It follows from Lemma \ref{lem:inverseMR} that
\begin{equation}\label{AB-bounds}
  \|\cA(\zeta)\|_{\cB(L^2,L^2)} \leq C_U \frac{\e\,\theta_j(1+\beta)}{\nu} \leq \tfrac12 ,
  \qquad
  \|\cB(\zeta)\|_{\cB(L^2,L^2)} \leq 2 ,
  \qquad \zeta \in \Gamma_j ,
\end{equation}
and, by \eqref{Pi0-cancellation} and \eqref{estimatePi0cR},
\begin{equation}\label{Pi0A-bound}
  \|\Pi_0\cA(\zeta)\|_{\cB(L^2,\CC)} \leq \frac{2\e^3}{\nu}\|U\|_{L^2} .
\end{equation}
Note that $\cA(\zeta)$ and $\cB(\zeta)$ commute. By the definitions \eqref{auxiliaryRieszproj} and \eqref{Rieszproj}, since $\cL_{\nu,\e,\beta}=\cM_{\nu,\e,\beta}-\im \e \cR_{\e,\beta}$, we observe that
\begin{align*}
    P_{\nu,\e,\beta}^{(j)} - Q_{\nu,\e,\beta}^{(j)}
    &= -\frac{\nu}{2\pi\im} \oint_{\Gamma_j} \Big((\cL_{\nu,\e,\beta} -\nu\zeta)^{-1} -(\cM_{\nu,\e,\beta} -\nu\zeta)^{-1} \Big)\drm \zeta \\
    &= -\frac{\im \nu \e}{2\pi\im} \oint_{\Gamma_j} (\cL_{\nu,\e,\beta} -\nu\zeta)^{-1} \cR_{\e,\beta}  (\cM_{\nu,\e,\beta} -\nu\zeta)^{-1} \drm \zeta \\
    &= -\frac{\nu}{2\pi\im} \oint_{\Gamma_j}  \cB(\zeta)\,\cA(\zeta)\, (\cM_{\nu,\e,\beta} -\nu\zeta)^{-1}\drm \zeta .
\end{align*}
Inserting the pole-free representation \eqref{cMresolvent} of $(\cM_{\nu,\e,\beta}-\nu\zeta)^{-1}$ and separating the $\Pi_{\neq}$ and $\Pi_0$ parts, we obtain \eqref{cT-decomp}. 
It remains to estimate the four quantities in \eqref{bd:Pi0F1F2}. Recall that the length of $\Gamma_j$ is $\pi$, and from Lemma \ref{lem:resolventD} that $|\zeta-\mu_0|^{-1}\leq 2$ and $\|(\cD_\e-\zeta)^{-1}\Pi_{\neq}\|\leq2\theta_j\leq2$ on $\Gamma_j$.

(1) \emph{Estimate of $\cK_{\nu,\e,\beta}^{(j)}$.} Directly from \eqref{def-cTj}, \eqref{AB-bounds} and \eqref{estimate-D},
\begin{equation}\label{estimate-T}
   \|\cK_{\nu,\e,\beta}^{(j)}\|_{\cB(L^2,L^2)} \leq \tfrac12 \cdot 2 \cdot C_U\frac{\e\theta_j(1+\beta)}{\nu}\cdot 2\theta_j \leq C_U \frac{\e\,\theta_j^2(1+\beta)}{\nu} \leq C_U \frac{\e\,(1+\beta)}{\nu} .
\end{equation}
Since $\cA$ and $\cB$ commute, $\Pi_0\cB\cA = \Pi_0\cA\,\cB$, so \eqref{Pi0A-bound} gives in the same way
\begin{equation}\label{estimate-Pi0T}
  \|\Pi_0 \cK_{\nu,\e,\beta}^{(j)}\|_{\cB(L^2,\CC)} \leq \tfrac12\cdot \frac{2\e^3}{\nu}\|U\|_{L^2}\cdot 2\cdot 2\theta_j \leq C_U \frac{\e^3}{\nu} .
\end{equation}

(2) \emph{Estimate of $W_{\nu,\e,\beta}^{(j)}$.} By \eqref{def-h} and \eqref{estimate-D} we have, for $\zeta\in\Gamma_j$,
\begin{equation}\label{h-bound}
  \|h_{\nu,\e}(\zeta,\cdot)\|_{L^2} \leq \frac{2\theta_j\|U''\|_{L^2}}{\e\nu}+1 .
\end{equation}
Hence, from \eqref{def-Wj}, \eqref{AB-bounds} and $|\zeta-\mu_0|^{-1}\leq2$,
\begin{align}\label{estimate-W}
  \|W_{\nu,\e,\beta}^{(j)}\|_{L^2}
  &\leq \tfrac12\cdot 2\cdot 2\cdot C_U\frac{\e\theta_j(1+\beta)}{\nu}\left(\frac{2\theta_j\|U''\|_{L^2}}{\e\nu}+1\right) \nonumber\\
  &\leq C_U\left( \frac{\theta_j^2(1+\beta)}{\nu^2} + \frac{\e\,\theta_j(1+\beta)}{\nu} \right)
  \leq C_U\,\frac{1+\beta}{\nu^2} ,
\end{align}
where in the last step we used $\theta_j\leq1$ and $\e\nu\leq1$. Likewise, using $\Pi_0\cB\cA=\Pi_0\cA\,\cB$ together with \eqref{Pi0A-bound} and \eqref{h-bound},
\begin{equation}\label{estimate-Pi0W}
  |\Pi_0 W_{\nu,\e,\beta}^{(j)}| \leq \tfrac12 \cdot 2 \cdot \frac{2\e^3}{\nu}\|U\|_{L^2}\cdot 2 \cdot \left(\frac{2\theta_j\|U''\|_{L^2}}{\e\nu}+1\right)
  \leq C_U \frac{\e^2}{\nu^2}\big(\theta_j+\e\nu\big) \leq C_U\frac{\e^2}{\nu^2} ,
\end{equation}
since $\theta_j+\e\nu \leq 1+\e\nu \leq 2$. Collecting \eqref{estimate-T}, \eqref{estimate-Pi0T}, \eqref{estimate-W} and \eqref{estimate-Pi0W} gives \eqref{bd:Pi0F1F2}.
\end{proof}


\bigskip

\section{Kato Reduction}\label{sec Kato}

This section is devoted to constructing an isomorphism between the ranges of $Q_{\nu,\e,\beta}^{(j)}$ and $P_{\nu,\e,\beta}^{(j)}$. We employ the transformation operator introduced in \cite{Colombo}:
\begin{equation}\label{Katomorphism}
    \cU_{\nu,\e,\beta}^{(j)} \coloneqq
    \uno - P_{\nu,\e,\beta}^{(j)} - Q_{\nu,\e,\beta}^{(j)} .
\end{equation}
It is readily verified that $\cU_{\nu,\e,\beta}^{(j)}$ maps $\mathrm{Ran}\, Q_{\nu,\e,\beta}^{(j)}$ onto $\mathrm{Ran}\, P_{\nu,\e,\beta}^{(j)}$ and conversely. Our objective is to
show that, under suitable assumptions, $\cU_{\nu,\e,\beta}^{(j)}$ provides an isomorphism between these two subspaces. To this end, we formally write the inverse of $\cU_{\nu,\e,\beta}^{(j)}$
as
\begin{equation}
    [\cU_{\nu,\e,\beta}^{(j)}]^{-1}
    =
    \Big( \uno -
    \big(P_{\nu,\e,\beta}^{(j)}-Q_{\nu,\e,\beta}^{(j)}\big)^2
    \Big)^{-1}
    \big(\uno-P_{\nu,\e,\beta}^{(j)}-Q_{\nu,\e,\beta}^{(j)}\big).
\end{equation}
Therefore, $\cU_{\nu,\e,\beta}^{(j)}$ is invertible whenever $\uno-\big(P_{\nu,\e,\beta}^{(j)}-Q_{\nu,\e,\beta}^{(j)}\big)^2$ is invertible. Furthermore, Lemma~\ref{lem:Spectrum-M} implies that the $j-$th eigenspaces of $\cM_{\nu,\e,\beta}$, namely, $\mathrm{Ran}\,Q_{\nu,\e,\beta}^{(j)}$, are finite-dimensional. Hence, $\cU_{\nu,\e,\beta}^{(j)}$ defines an isomorphism between $\mathrm{Ran}\,Q_{\nu,\e,\beta}^{(j)}$ and $\mathrm{Ran}\,P_{\nu,\e,\beta}^{(j)}$. For sufficiently small $\nu^{-1}\e$, the invertibility of $\uno-\big(P_{\nu,\e,\beta}^{(j)}-Q_{\nu,\e,\beta}^{(j)}\big)^2$ follows from
the following lemma.
\begin{Lemma}\label{lem:KatoIso}
Let $\beta\geq0$, $0<\nu\leq1$ and $0<\e<1$. Then there exists a constant $\delta_{U}>0$, depending on $\|U\|_{C^2}$ but independent of $\e$, $\nu$, $j$ and $\beta$, such that, whenever $\e(1+\beta) / \nu <\delta_{U}$, the operator
\[
\uno-\big(P_{\nu,\e,\beta}^{(j)}-Q_{\nu,\e,\beta}^{(j)}\big)^2,
\]
with $Q_{\nu,\e,\beta}^{(j)}$ and $P_{\nu,\e,\beta}^{(j)}$ defined in \eqref{auxiliaryRieszproj} and \eqref{Rieszproj}, respectively, is invertible for all $j\in\NN$. Therefore, the operator $\cU_{\nu,\e,\beta}^{(j)}$ defined in \eqref{Katomorphism} is an isomorphism.
\end{Lemma}
\begin{proof}
Since condition \eqref{smallcond} holds whenever $\frac{\e}{\nu}(1+\beta)<\delta_{U}$ for sufficiently small $\delta_{U}$, we may apply Lemma~\ref{lem:P-Q} to decompose
\begin{equation}\label{cT}
     P_{\nu,\e,\beta}^{(j)} - Q_{\nu,\e,\beta}^{(j)} = \cK_{\nu,\e,\beta}^{(j)}  + W_{\nu,\e,\beta}^{(j)} \Pi_0 .
\end{equation}
By \eqref{bd:Pi0F1F2}, $\cK^{(j)}_{\nu,\e,\beta}$ is uniformly bounded in $j,\e$, and $\nu$, and satisfies
\begin{equation}\label{boundIpmcT}
   \| \cK^{(j)}_{\nu,\e,\beta} \|_{\cB(L^2,L^2)} \leq  C_{U}\frac{\e}{\nu}(1+\beta) .
\end{equation}
Choosing $\delta_{U}$ sufficiently small, such that $C_{U}\delta_{U} <1$ and \eqref{smallcond} holds, it follows that the operators $\uno \pm \cK_{\nu,\e,\beta}^{(j)}$ are invertible, with inverses given by a Neumann series and satisfying
\begin{equation}
    \| (\uno \pm  \cK_{\nu,\e,\beta}^{(j)})^{-1} \|_{\cB(L^2,L^2)} \leq \frac{1}{1-C_{U} \delta_{U}} .
\end{equation}
We now consider
\begin{equation}\label{prodottonotevole(A^2-B^2)}
    \uno - (P_{\nu,\e,\beta}^{(j)} - Q_{\nu,\e,\beta}^{(j)} )^2 = \big(  \uno - (P_{\nu,\e,\beta}^{(j)} - Q_{\nu,\e,\beta}^{(j)} )\big) \big(  \uno + (P_{\nu,\e,\beta}^{(j)} - Q_{\nu,\e,\beta}^{(j)} ) \big)\, .
\end{equation}
Using \eqref{cT}, each factor can be written as a rank-one update of an invertible operator:
\begin{equation}
       \uno \pm (P_{\nu,\e,\beta}^{(j)} - Q_{\nu,\e,\beta}^{(j)} ) = \uno \pm \cK_{\nu,\e,\beta}^{(j)}  \pm W_{\nu,\e,\beta}^{(j)} \Pi_0.
\end{equation}
We then apply  Lemma \ref{lem:ShermanMorrisoninf}, with
\[
    H=L^2(\TT)\,,\quad \cA = \uno \pm \cK_{\nu,\e,\beta}^{(j)}\, ,\quad f= \pm   W_{\nu,\e,\beta}^{(j)}\quad\text{and}\quad g=1 ,
\]
which yields
\begin{equation}\label{tojustify}
\begin{split}
    \big( \uno \pm (P_{\nu,\e,\beta}^{(j)} - Q_{\nu,\e,\beta}^{(j)} ) \big)^{-1} & \stackrel{\eqref{originalShermanMorrison}} {=}\big( \uno \pm \cK_{\nu,\e,\beta}^{(j)}) \big)^{-1} \\
    & \qquad \quad \mp \frac{ \big( \uno \pm \cK_{\nu,\e,\beta}^{(j)} \big)^{-1} W_{\nu,\e,\beta}^{(j)}(y) \Pi_0 \big( \uno \pm \cK_{\nu,\e,\beta}^{(j)} \big)^{-1}}{1\pm \Pi_0  \big( \uno \pm \cK_{\nu,\e,\beta}^{(j)} \big)^{-1} W_{\nu,\e,\beta}^{(j)}(y) }\, .
\end{split}
\end{equation}
To apply Lemma~\ref{lem:ShermanMorrisoninf}, we need to check
\[
\big|\Pi_0  \big( \uno \pm \cK_{\nu,\e,\beta}^{(j)} \big)^{-1}  W_{\nu,\e,\beta}^{(j)}(y) \big|\neq 1 .
\]
Using \eqref{bd:Pi0F1F2} together with \eqref{boundIpmcT}, we estimate
\begin{align*}
    &|\Pi_0  \big( \uno \pm \cK_{\nu,\e,\beta}^{(j)} \big)^{-1}  W_{\nu,\e,\beta}^{(j)}| \\
    & = \left| \Pi_0 W_{\nu,\e,\beta}^{(j)} +\sum_{n=1}^\infty (\mp 1)^n\Pi_0 (\cK_{\nu,\e,\beta}^{(j)})^n W_{\nu,\e,\beta}^{(j)} \right| \\
    &\leq |\Pi_0 W_{\nu,\e,\beta}^{(j)}| + \| \Pi_0\cK^{(j)}_{\nu,\e,\beta} \|_{\cB(L^2,\CC)} \| (\uno \pm  \cK_{\nu,\e,\beta}^{(j)})^{-1} \|_{\cB(L^2,L^2)} \|W_{\nu,\e,\beta}^{(j)}\|_{L^2} \\
    & \leq  C_{U} \frac{\e^2}{\nu^2} + C_{U} \frac{\e^3}{\nu}\cdot\frac{1}{1-C_U\delta_U}\cdot\frac{1+\beta}{\nu^2}
    \leq C_U \frac{\e^2}{\nu^2}\left( 1 + \frac{\e(1+\beta)}{\nu}\cdot\frac{1}{1-C_U\delta_U}\right),
\end{align*}
where we used \eqref{bd:Pi0F1F2}. Hence, for $\delta_{U}$ sufficiently small, we obtain
\begin{align}
     \left| \Pi_0  \big( \uno \pm \cK_{\nu,\e,\beta}^{(j)} \big)^{-1} W_{\nu,\e,\beta}^{(j)} \right| \le C_{U} \delta_{U}^2 \left(1 + \frac{\delta_{U}}{1-C_{U} \delta_{U}} \right) <1 ,
\end{align}
since $\frac{\e}{\nu}\leq\frac{\e}{\nu}(1+\beta)<\delta_U$ and $\nu\leq1$ give $\frac{\e^2}{\nu^2}<\delta_U^2$.
This justifies \eqref{tojustify}. Consequently, $\uno - (P_{\nu,\e,\beta}^{(j)} - Q_{\nu,\e,\beta}^{(j)})^2$ is invertible, and the operator $\cU_{\nu,\e,\beta}^{(j)}$ in \eqref{Katomorphism} is bounded with bounded inverse.
\end{proof}
We now obtain the final conclusion of Kato's method. Throughout, let $\beta\geq0$ and $\e,\nu>0$ satisfy \eqref{smallcond}, so that by Lemma \ref{lem:KatoIso} the operator $\cU_{\nu,\e,\beta}^{(j)}$ defined in \eqref{Katomorphism} is an isomorphism. We first note that the projections $Q_{\nu,\e,\beta}^{(j)}$ act as the identity on the corresponding eigenspaces of $\cM_{\nu,\e,\beta}$. Therefore, using the definition $\cU_{\nu,\e,\beta}^{(j)} =\uno-P_{\nu,\e,\beta}^{(j)}-Q_{\nu,\e,\beta}^{(j)}$, we obtain
\begin{equation}\label{Uproj-0}
    \cU_{\nu,\e,\beta}^{(0)}
    [F_{\nu,\e,\beta}(y)-\im\nu\e]
    =
    -P_{\nu,\e,\beta}^{(0)}
    [F_{\nu,\e,\beta}(y)-\im\nu\e].
\end{equation}
Since the dimension of $\mathrm{Ran}\,P_{\nu,\e,\beta}^{(0)}$ is one, we may fix, once and for all, the choice of eigenvector representative
\begin{equation}\label{V-def-0}
    V_{\nu,\e,\beta}^{(0)}
    \coloneqq
    P_{\nu,\e,\beta}^{(0)}
    [F_{\nu,\e,\beta}(y)-\im\nu\e]
    =
    -\cU_{\nu,\e,\beta}^{(0)}
    [F_{\nu,\e,\beta}(y)-\im\nu\e],
\end{equation}
and similarly
\begin{equation}
    V_{\nu,\e,\beta}^{(\pm j)}
    \coloneqq
    P_{\nu,\e,\beta}^{(j)}e^{\pm \im j y}
    =
    -\cU_{\nu,\e,\beta}^{(j)} e^{\pm \im j y},
    \quad \forall j\in\ZZ^+.
\end{equation}
Let $\cV_{\nu,\e,\beta}^{(j)} \coloneqq \mathrm{Ran}\,P_{\nu,\e,\beta}^{(j)}$, $j\in\NN$. By the spectral decomposition associated with the Riesz projections, we have
\begin{equation}
  L^2(\TT)
  =
  \bigoplus_{j\in\NN}\cV_{\nu,\e,\beta}^{(j)}
\end{equation}
and
\begin{equation}
  \mathrm{dim}(\cV_{\nu,\e,\beta}^{(j)})
  =
  \mathrm{dim}(\mathrm{Ran}\,Q_{\nu,\e,\beta}^{(j)})
  =
  \begin{cases}
    1, & \mbox{if } j=0, \\
    2, & \mbox{if } j\neq 0.
  \end{cases}
\end{equation}
Since $\cL_{\nu,\e,\beta}$ commutes with the projections $P_{\nu,\e,\beta}^{(j)}$, each subspace $\cV_{\nu,\e,\beta}^{(j)}$ is invariant under $\cL_{\nu,\e,\beta}$. Consequently, for $j\in\ZZ^+$, $\cL_{\nu,\e,\beta}$ has either two distinct eigenvalues or a single eigenvalue of algebraic multiplicity two with associated eigenspace $\cV_{\nu,\e,\beta}^{(j)}$. For $j=0$, the operator $\cL_{\nu,\e,\beta}$ has a simple eigenvalue with associated eigenspace $\cV_{\nu,\e,\beta}^{(0)}$.

It remains to show that these eigenvalues exhaust $\sigma_{L^2}(\cL_{\nu,\e,\beta})$.
\begin{Proposition}\label{prop:exhaust}
Assume \eqref{smallcond}. Then $\cL_{\nu,\e,\beta}$ has compact resolvent, its spectrum consists of isolated eigenvalues of finite algebraic multiplicity, and
\begin{equation}\label{spectrum-localized}
  \sigma_{L^2}(\cL_{\nu,\e,\beta}) \subset \bigcup_{j\in\NN} \nu\,\overline{D\big(\mu_j,\tfrac12\big)} .
\end{equation}
Moreover, each disk $\nu\overline{D(\mu_j,\tfrac12)}$ contains exactly $\dim\cV_{\nu,\e,\beta}^{(j)}$ eigenvalues counted with algebraic multiplicity, namely $1$ if $j=0$, and $2$ if $j \geq 1$. Hence
\[
  \sigma_{L^2}(\cL_{\nu,\e,\beta}) = \big\{\lambda^{(0)}_{\nu,\e,\beta}\big\} \cup \big\{\lambda^{(\pm j)}_{\nu,\e,\beta}\big\}_{j \in \ZZ^+} .
\]
\end{Proposition}
\begin{proof}
By Lemma \ref{lem:resolventOutside}, $\cL_{\nu,\e,\beta}-\nu\zeta$ is boundedly invertible for every $\zeta\in\Omega_{\nu,\e,\beta}$, that is for every $\zeta$ outside $\bigcup_j\overline{D(\mu_j,\tfrac12)}$. Since $\lambda=\nu\zeta$, this is exactly \eqref{spectrum-localized}. In particular $\rho_{L^2}(\cL_{\nu,\e,\beta})\neq\emptyset$, and for $\zeta\in\Omega_{\nu,\e,\beta}$ the identity \eqref{resolventcL1} exhibits $(\cL_{\nu,\e,\beta}-\nu\zeta)^{-1}$ as the product of a bounded operator with $(\cM_{\nu,\e,\beta}-\nu\zeta)^{-1}$, which is compact by \eqref{cMresolvent}, being the sum of a rank-one operator and $\frac1\nu(\cD_\e-\zeta)^{-1}\Pi_{\neq}$, a Fourier multiplier with symbol tending to $0$. Hence $\cL_{\nu,\e,\beta}$ has compact resolvent and its spectrum is a discrete set of eigenvalues of finite algebraic multiplicity.

For the counting, fix $j$ and recall from \eqref{Rieszproj} that $P^{(j)}_{\nu,\e,\beta}$ is the Riesz projection of $\cL_{\nu,\e,\beta}$ associated with the part of the spectrum enclosed by $\nu\Gamma_j$, so that the number of eigenvalues of $\cL_{\nu,\e,\beta}$ inside $\nu\Gamma_j$, counted with algebraic multiplicity, equals $\operatorname{rank}P^{(j)}_{\nu,\e,\beta}=\dim\cV^{(j)}_{\nu,\e,\beta}$. By Lemma \ref{lem:KatoIso} the operator $\cU^{(j)}_{\nu,\e,\beta}$ is an isomorphism from $\mathrm{Ran}\,Q^{(j)}_{\nu,\e,\beta}$ onto $\mathrm{Ran}\,P^{(j)}_{\nu,\e,\beta}$, whence $\dim\cV^{(j)}_{\nu,\e,\beta}=\operatorname{rank}Q^{(j)}_{\nu,\e,\beta}$, which by Lemma \ref{lem:Spectrum-M} equals $1$ for $j=0$ and $2$ for $j\geq1$. Since the disks $\overline{D(\mu_j,\frac{1}{2})}$ are pairwise disjoint by Lemma \ref{lem:resolventD}, summing over $j$ accounts for all of $\sigma_{L^2}(\cL_{\nu,\e,\beta})$ by \eqref{spectrum-localized}.
\end{proof}

\bigskip

\section{The Expansion of $\lambda_{\nu,\e,\beta}^{(0)}$}\label{sec expansion}

In this section, we analyze the stability of the eigenvalues of $\cL_{\nu,\e,\beta}$. Let $\lambda_{\nu,\e,\beta}^{(\pm j)}$ denote the eigenvalues associated with the eigenspaces $\cV_{\nu,\e,\beta}^{(j)}$. Since $\lambda_{\nu,\e,\beta}^{(\pm j)}$ are contained in the circles $\nu\Gamma_j$, we have
\begin{equation}\label{eig-localization}
  \big|\lambda_{\nu,\e,\beta}^{(\pm j)} - \nu\mu_j\big| < \frac{\nu}{2}, \quad j \in \NN ,
\end{equation}
that is, $$|\lambda_{\nu,\e,\beta}^{(\pm j)}+\nu(j^2+\e^2)|<\frac{\nu}{2}$$ for $j \in \ZZ^+$, while for $j=0$ the centre of the circle is the non-real number $\nu\mu_0 = -\nu\e^2+\frac{\im\beta}{\e}$. The estimate \eqref{eig-localization} immediately implies the stability of $\lambda_{\nu,\e,\beta}^{(\pm j)}$ for $j\in\ZZ^+$, since $$\Re\,\lambda_{\nu,\e,\beta}^{(\pm j)}<-\nu(j^2+\e^2)+\frac\nu2\leq-\frac\nu2.$$
 Together with Proposition \ref{prop:exhaust}, we prove the first part of Theorem \ref{thm:main-1}. Therefore, it remains to determine the stability of the eigenvalue $\lambda_{\nu,\e,\beta}^{(0)}$.

To establish this, we derive an asymptotic expansion for $\lambda_{\nu,\e,\beta}^{(0)}$. Throughout the remainder of this section, since we only consider $\lambda_{\nu,\e,\beta}^{(0)}$, we suppress the superscript $(0)$ for simplicity of notation, and all the estimates are made under the single standing assumption \eqref{smallcond}, that is
\[
\frac{\e}{\nu}(1+\|U\|_{C^2})^2(1+\beta) < \frac{1}{20}, \qquad 0<\nu\leq1, \quad 0<\e<1, \quad \beta \geq 0 .
\]
No lower bound on $\beta$ is imposed at any point.

\subsection{Elementary estimates}
We first record the estimates that will be used throughout, all of them uniform in $\beta \geq 0$. Recall $b=\frac{\beta}{\e\nu}$, $\theta=\min\{1,b^{-1}\}$, and $\Theta=\theta+\e\nu$ from \eqref{def-b} \eqref{def-theta}, and \eqref{def-Theta}. Recall the notation $\cO_{\rm op}(\cdot)$ and $\cO_{\rm fun}(\cdot)$ introduced in \eqref{def-O-notation}.

By \eqref{bd-Dnueb},
\begin{equation}\label{exp:D-epsbeta}
  \cD_{\nu,\e,\beta}^{-m}\Pi_{\neq}
  = \cO_{\rm op}\big(\theta^{m}\big), \qquad m\in\NN,
\end{equation}
and therefore, since $U''$ has zero average,
\begin{equation}\label{exp-F-order}
  F_{\nu,\e,\beta} = \cD_{\nu,\e,\beta}^{-1}\Pi_{\neq}(U'')
  = \cO_{\rm fun}\big(\theta\big),
  \qquad
  V_0 = F_{\nu,\e,\beta}-\im\e\nu = \cO_{\rm fun}\big(\Theta\big).
\end{equation}
Because $\|(\partial_y+\e)^{-1}\Pi_{\neq}\|_{\cB(L^2,L^2)}\leq1$, the scalar quantity that will produce the leading term satisfies
\begin{equation}\label{exp-UF}
  \cI_{\nu,\e,\beta} \coloneqq \langle (\de_y+\e)^{-1} U, (\de_y+\e)^{-1} F_{\nu,\e,\beta} \rangle_{L^2}
  = \cO \big(\theta\big).
\end{equation}
Furthermore, since $\e<1$ and $\nu\leq1$,
\begin{equation}\label{exp-numu0}
  \nu \mu_0
  = -\nu \e^2 + \frac{\im \beta}{\e} = \cO \Big(\frac{\nu}{\theta}\Big),
  \qquad
  |\nu\zeta| \leq \nu\big(\e^2+b+\tfrac12\big) \leq 3\,\frac{\nu}{\theta}
  \quad \text{for } \zeta\in\Gamma_0.
\end{equation}
Finally, we need a lower bound for $\|V_0\|_{L^2}$, which is where the scale $\Theta$ originates.
\begin{Lemma}\label{lem:V0}
If $U\not\equiv0$, then
\begin{equation}\label{V0-two-sided}
  c_U\,\Theta \;\leq\; \|V_0\|_{L^2} \;\leq\; C_U\,\Theta ,
\end{equation}
for constants $0<c_U\leq C_U$ depending only on $U$. If $U \equiv 0$, $\|V_0\|_{L^2} =\e\nu$.
\end{Lemma}
\begin{proof}
Since $\Pi_0V_0=-\im\e\nu$ and $\Pi_{\neq}V_0=F_{\nu,\e,\beta}$,  Parseval's identity and $U''_j=-j^2U_j$ give
\begin{equation}\label{V0-Fourier}
  \|V_0\|_{L^2}^2 = \e^2\nu^2 + \sum_{j\neq0}\frac{|U''_j|^2}{j^4+b^2}
  = \e^2\nu^2 + \sum_{j\neq0}\frac{j^4|U_j|^2}{j^4+b^2} .
\end{equation}
Note that $\frac{j^4}{j^4+b^2}\leq\min\{j^4,j^4b^{-2}\} = j^4\theta^2$, so the sum is bounded above by $\|U''\|_{L^2}^2\theta^2$, and thus
\[
\|V_0\|_{L^2}^2 \leq \e^2\nu^2 + \sum_{j\neq0} j^4 \theta^2 |U_j|^2 = \e^2\nu^2 + \|U''\|_{L^2}^2\theta^2 .
\]
On the other hand,
\[
\|V_0\|_{L^2}^2 = \e^2\nu^2 + \sum_{j\neq0}\frac{j^4|U_j|^2}{j^4+b^2} \geq \e^2\nu^2 + \sum_{j \neq 0}\frac{|U_j|^2}{1+b^2} \geq \e^2\nu^2 + \frac12\|U\|_{L^2}^2\theta^2
\]
Combining these two inequalities above, we get \eqref{V0-two-sided}.
If $U\equiv 0$, then $V_0=-\im\e\nu$ and $\|V_0\|_{L^2}^2 =\e^2\nu^2$.
\end{proof}

\subsection{The expansion of the eigenvector}
Recalling \eqref{Uproj-0} and \eqref{V-def-0}, we begin with the eigenvalue equation
\begin{equation}
\lambda_{\nu,\e,\beta}V_{\nu,\e,\beta}
=\cL_{\nu,\e,\beta}V_{\nu,\e,\beta}
=\cL_{\nu,\e,\beta}P_{\nu,\e,\beta}\,V_0 ,
\end{equation}
with $V_0=F_{\nu,\e,\beta}-\im\e\nu$ as in \eqref{exp-F}. Taking the $L^2(\TT)$ inner product with $V_{\nu,\e,\beta}$ yields
\begin{equation}\label{eq:trivlambda1}
\lambda_{\nu,\e,\beta}
=\frac{1}{\|V_{\nu,\e,\beta}\|^2_{L^2}}
\langle V_{\nu,\e,\beta}, \cL_{\nu,\e,\beta}V_{\nu,\e,\beta} \rangle_{L^2}.
\end{equation}
Exactly as before, since the circulation of the identity vanishes,
\begin{align*}
  \cL_{\nu,\e,\beta} P_{\nu,\e,\beta}
  &= -\frac{\nu^2}{2\pi\im} \oint_{\Gamma_0}
  \zeta
  \big(\uno-\im\e(\cM_{\nu,\e,\beta}-\nu\zeta)^{-1}\cR_{\e,\beta}\big)^{-1}
  (\cM_{\nu,\e,\beta}-\nu\zeta)^{-1}
  \drm\zeta ,
\end{align*}
and by Lemma~\ref{lem:Spectrum-M},
\begin{equation}
    (\cM_{\nu,\e,\beta}-\nu\zeta)^{-1} V_0
    =-\frac{1}{\nu(\zeta-\mu_0)} V_0 .
\end{equation}
Consequently
\begin{equation}\label{V-and-LV}
\begin{split}
  & V_{\nu,\e,\beta}
  =\frac{1}{2\pi\im}\oint_{\Gamma_0}
  \frac{\sum_{p\geq0}\big(\cT_1(\zeta)+\cT_2(\zeta)\big)^p}{\zeta-\mu_0}
  \,V_0\,
  \drm\zeta ,
  \\
  &\cL_{\nu,\e,\beta}V_{\nu,\e,\beta}
  =\frac{1}{2\pi\im}\oint_{\Gamma_0}
  \frac{\nu\zeta\,\sum_{p\geq0}\big(\cT_1(\zeta)+\cT_2(\zeta)\big)^p}{\zeta-\mu_0}
  \,V_0\,
  \drm\zeta ,
\end{split}
\end{equation}
where $\cT_1(\zeta)$ and $\cT_2(\zeta)$ are the two pieces of $\im\e(\cM_{\nu,\e,\beta}-\nu\zeta)^{-1}\cR_{\e,\beta}$ identified in \eqref{auxiteration1}, namely,
\begin{equation}\label{def-T1T2}
  \cT_1(\zeta) = \frac{\im\e}{\nu}(\cD_\e-\zeta)^{-1}\Pi_{\neq}\cR_{\e,\beta},
  \qquad
  \cT_2(\zeta) = \frac{\im\e\, h_{\nu,\e}(\zeta,y)}{\nu(\zeta-\mu_0)}\Pi_0\cR_{\e,\beta} .
\end{equation}
The two expansions in \eqref{V-and-LV} differ only by the factor $\nu\zeta$ in the integrand. This simple observation organizes the whole computation, and we isolate it.
\begin{Lemma}\label{lem:pole-order}
Let $\Phi$ be an $L^2(\TT)$-valued function, holomorphic in a neighborhood of the closed disk bounded by $\Gamma_0$. Then
\begin{equation}\label{pole-order-simple}
  \frac{1}{2\pi\im}\oint_{\Gamma_0}\frac{\nu\zeta\,\Phi(\zeta)}{\zeta-\mu_0}\drm\zeta
  \;=\; \nu\mu_0 \cdot \frac{1}{2\pi\im}\oint_{\Gamma_0}\frac{\Phi(\zeta)}{\zeta-\mu_0}\drm\zeta ,
\end{equation}
whereas for a double pole
\begin{equation}\label{pole-order-double}
  \frac{1}{2\pi\im}\oint_{\Gamma_0}\frac{\nu\zeta\,\Phi(\zeta)}{(\zeta-\mu_0)^2}\drm\zeta
  \;=\; \nu\mu_0 \cdot \frac{1}{2\pi\im}\oint_{\Gamma_0}\frac{\Phi(\zeta)}{(\zeta-\mu_0)^2}\drm\zeta
  \;+\; \nu\,\Phi(\mu_0) .
\end{equation}
\end{Lemma}
\begin{proof}
Both sides of \eqref{pole-order-simple} equal $\nu\mu_0\Phi(\mu_0)$ by the residue theorem. For \eqref{pole-order-double}, the residue at the double pole is $$\partial_\zeta\big[\nu\zeta\Phi(\zeta)\big]_{\zeta=\mu_0}=\nu\Phi(\mu_0)+\nu\mu_0\partial_\zeta\Phi(\mu_0),$$
 while the right-hand side is $\nu\mu_0\partial_\zeta\Phi(\mu_0)+\nu\Phi(\mu_0)$.
\end{proof}

Lemma \ref{lem:pole-order} has the following consequence, which we use as a guide throughout. Since, by \eqref{eq:trivlambda1}, the eigenvalue is a quotient of the two integrals in \eqref{V-and-LV}, any term of the Neumann expansion whose integrand has only a simple pole at $\mu_0$ contributes to $\cL_{\nu,\e,\beta}V_{\nu,\e,\beta}$ exactly $\nu\mu_0$ times its contribution to $V_{\nu,\e,\beta}$, and therefore cannot influence $\lambda_{\nu,\e,\beta}$ beyond reproducing $\nu\mu_0$. All the nontrivial information comes from terms whose pole at $\mu_0$ has order at least two, that is, by \eqref{def-T1T2}, from terms containing the zero-mode factor $\cT_2$, and, as we shall see, up to second order exactly one such term survives.

It is essential for this argument that $\mu_0$ be the only singularity of the integrands inside $\Gamma_0$. This is guaranteed by Lemma \ref{lem:ResolventcM}: both $(\cD_\e-\zeta)^{-1}\Pi_{\neq}$ and $h_{\nu,\e}(\zeta,\cdot)$ are analytic on $\{\Re\,\zeta>-1\}\supset\Gamma_0$ and its interior, whatever the value of $\beta\geq0$. In particular, no residue computed below depends on the position of $-\e^2$ relative to $\Gamma_0$, which is what makes the whole expansion uniform in $\beta$. 

By \eqref{lem:inverseMR-term-1}, \eqref{lem:inverseMR-term-2} and the two observations in the proof of Lemma \ref{lem:inverseMR}, applied with $j=0$ so that $\theta_j=\theta$, and noting Remark \ref{rem:contour}, 
the orders of $\cT_1$ and $\cT_2$ on $\Gamma_0$ are
\begin{equation}\label{estimate-T1,2}
\cT_1 = \cO_{\rm op}\big(\sigma\big), \qquad
\cT_2 = \cO_{\rm op}\big(\e^2\varrho\big),
\qquad
\sigma \coloneqq \frac{\e\,\theta(1+\beta)}{\nu} , \quad
\varrho \coloneqq \frac{\theta}{\nu^2}+\frac{\e}{\nu} ,
\end{equation}
and, since
\begin{equation}\label{T2-vs-T1}
  \frac{\e^2\varrho}{\sigma} = \frac{\e}{\nu(1+\beta)} + \frac{\e^2}{\theta(1+\beta)} \leq \frac{2\e}{\nu} ,
\end{equation}
(the second term being estimated as in the proof of Lemma \ref{lem:inverseMR}), we have $\cT_2 = \cO_{\rm op}(\frac{\e}{\nu}\sigma)$. In particular $\sigma\leq C_U\frac{\e}{\nu}(1+\beta)$ is small by \eqref{smallcond}.

Expanding the Neumann series in \eqref{V-and-LV} up to second order, we write
\begin{equation}\label{VandvSplit}
  V_{\nu,\e,\beta} = V_0+V_1+V_2+V_r ,
  \qquad
  \cL_{\nu,\e,\beta}V_{\nu,\e,\beta} = v_0+v_1+v_2+v_r ,
\end{equation}
where $V_p$, respectively $v_p$, denotes the contribution of the $p$-th power of $\cT_1+\cT_2$, and $V_r$, $v_r$ collect all terms with $p\geq3$. By \eqref{estimate-T1,2} and \eqref{exp-numu0},
\begin{equation}\label{est-Vr}
  V_r = \cO_{\rm fun}\big(\Theta\,\sigma^3\big),
  \qquad
  v_r = \cO_{\rm fun}\Big(\frac{\nu}{\theta}\,\Theta\,\sigma^3\Big).
\end{equation}

\subsection{Order zero and order one}
Since $\zeta\mapsto\frac{1}{\zeta-\mu_0}$ has a simple pole at $\mu_0$ with residue $1$,
\begin{equation}
  V_0 = \frac{1}{2\pi\im}\oint_{\Gamma_0}\frac{V_0}{\zeta-\mu_0}\drm\zeta ,
  \qquad
  v_0 = \frac{1}{2\pi\im}\oint_{\Gamma_0}\frac{\nu\zeta\,V_0}{\zeta-\mu_0}\drm\zeta = \nu\mu_0 V_0 ,
\end{equation}
so that $v_0=\nu\mu_0V_0$.

At order one, we treat the contributions of $\cT_1$ and $\cT_2$ separately. The factor $(\cD_\e-\zeta)^{-1}\Pi_{\neq}$ in $\cT_1$ is analytic at $\mu_0$, so the corresponding integrand has a simple pole and, since $\cD_\e-\mu_0=\cD_{\nu,\e,\beta}$,
\begin{equation}\label{V11}
\begin{split}
  &V_{1,1} \coloneqq \frac{1}{2\pi\im}\oint_{\Gamma_0}\frac{\cT_1(\zeta)V_0}{\zeta-\mu_0}\drm\zeta
  = \frac{\im\e}{\nu}\,\cD_{\nu,\e,\beta}^{-1}\Pi_{\neq}\cR_{\e,\beta}V_0 ,
  \\
  &\frac{1}{2\pi\im}\oint_{\Gamma_0}\frac{\nu\zeta\,\cT_1(\zeta)V_0}{\zeta-\mu_0}\drm\zeta = \nu\mu_0 V_{1,1} .
\end{split}
\end{equation}
The factor $\cT_2$ carries an extra $(\zeta-\mu_0)^{-1}$, so its integrand has a double pole at $\mu_0$, and the residues must be computed separately. Using $\Pi_0\cR_{\e,\beta}V_0 = \e^2\cI_{\nu,\e,\beta}$, which follows from \eqref{Cancelation}, \eqref{exp-UF} and $\Pi_{\neq}V_0=F_{\nu,\e,\beta}$, and using \eqref{h-at-mu0}, we obtain
\begin{align}
  V_{1,2} \coloneqq \frac{1}{2\pi\im}\oint_{\Gamma_0}\frac{\cT_2(\zeta)V_0}{\zeta-\mu_0}\drm\zeta
  &= \frac{\im\e^3}{\nu}\,\cI_{\nu,\e,\beta}\;\partial_\zeta h_{\nu,\e}(\mu_0,\cdot)
  = \frac{\e^2}{\nu^2}\,\cI_{\nu,\e,\beta}\;\cD_{\nu,\e,\beta}^{-2}\Pi_{\neq}(U'') , \label{V12}
\end{align}
while 
\begin{align}
  \frac{1}{2\pi\im}\oint_{\Gamma_0}\frac{\nu\zeta\,\cT_2(\zeta)V_0}{\zeta-\mu_0}\drm\zeta
  &= \im\e^3\,\cI_{\nu,\e,\beta}\; \partial_\zeta\big[\zeta\,h_{\nu,\e}(\zeta,\cdot)\big]\Big|_{\zeta=\mu_0} \nonumber\\
  &= \im\e^3\,\cI_{\nu,\e,\beta}\Big( h_{\nu,\e}(\mu_0,\cdot) + \mu_0\,\partial_\zeta h_{\nu,\e}(\mu_0,\cdot)\Big) \nonumber\\
  &= \frac{\e^2}{\nu}\,\cI_{\nu,\e,\beta}\,V_0 \; + \; \nu\mu_0 V_{1,2} . \label{v12}
\end{align}
Summing \eqref{V11}--\eqref{v12} and setting $V_1\coloneqq V_{1,1}+V_{1,2}$, we arrive at the identity
\begin{equation}\label{v1-identity}
  v_1 = \nu\mu_0\,V_1 \; + \; v_{\rm lead},
  \qquad
  v_{\rm lead} \coloneqq \frac{\e^2}{\nu}\,\cI_{\nu,\e,\beta}\,V_0 .
\end{equation}
The term $v_{\rm lead}$ is the only first-order contribution that is not proportional to $\nu\mu_0$, and it is the term that carries the leading correction to the eigenvalue. By \eqref{exp-UF} and \eqref{exp-F-order},
\begin{equation}\label{est-vlead}
  v_{\rm lead} = \cO_{\rm fun}\Big(\frac{\e^2\theta}{\nu}\,\Theta\Big),
  \qquad
  V_1 = \cO_{\rm fun}\big(\Theta\,\sigma\big) .
\end{equation}
\begin{Remark}\label{rem:cancellation}
It is worth comparing \eqref{V12}--\eqref{v12} with what one obtains from the raw formula \eqref{SMraw}. There $\cT_2$ is replaced by $\frac{u_{\e,\beta}(\zeta,y)}{\nu^2(\zeta-\mu_0)}\Pi_0\cR_{\e,\beta}$ and, in addition, $\cT_1$ acquires a zero-mode part $-\frac{\im\e}{\nu(\e^2+\zeta)}\Pi_0\cR_{\e,\beta}$. Splitting $u_{\e,\beta}=u_\e+\frac{\beta}{\e^2+\zeta}$, one then produces two extra terms, both proportional to $\beta^{-1}$, namely
\[
-\frac{\im\e^3}{\nu}\cI_{\nu,\e,\beta}\,\frac{1}{2\pi\im}\oint_{\Gamma_0}\frac{\drm\zeta}{(\zeta-\mu_0)(\zeta+\e^2)}
\quad\text{and}\quad
\frac{\e^2\beta}{\nu^2}\cI_{\nu,\e,\beta}\,\frac{1}{2\pi\im}\oint_{\Gamma_0}\frac{\drm\zeta}{(\zeta+\e^2)(\zeta-\mu_0)^2} .
\]
Their sum vanishes for every contour: using $\frac{\beta}{\nu}=\e b$ and $\zeta-\mu_0=(\zeta+\e^2)-\im b$, the two integrands combine to
\[
\frac{\e^2\,\cI_{\nu,\e,\beta}}{\nu(\zeta+\e^2)(\zeta-\mu_0)}\left[-\im\e+\frac{\e b}{\zeta-\mu_0}\right]
= -\frac{\im\e^3\,\cI_{\nu,\e,\beta}}{\nu\,(\zeta-\mu_0)^{2}} ,
\]
which has zero residue. Thus the two individually $\beta$-singular terms cancel identically, and the same computation with the extra factor $\nu\zeta$ leaves exactly the residue $-\im\e^3\cI_{\nu,\e,\beta}$, which is the $-\im\e\nu$ part of $v_{\rm lead}$ in \eqref{v1-identity}. Working with \eqref{cMresolvent}  rather than \eqref{SMraw} 
makes this cancellation automatic, which is why no negative power of $\beta$ occurs anywhere in the present section.
\end{Remark}

\subsection{Order two}
At second order we isolate the term containing $\cT_1$ twice. Every other term contains at least one factor $\cT_2$ and is therefore smaller by \eqref{T2-vs-T1}. Since $\cT_1$ does not raise the order of the pole at $\mu_0$,
\begin{equation}\label{V21}
  V_{2,1} \coloneqq \frac{1}{2\pi\im}\oint_{\Gamma_0}\frac{\cT_1(\zeta)^2V_0}{\zeta-\mu_0}\drm\zeta
  = -\frac{\e^2}{\nu^2}\big(\cD_{\nu,\e,\beta}^{-1}\Pi_{\neq}\cR_{\e,\beta}\big)^2 V_0 ,
\end{equation}
and the corresponding term in $\cL_{\nu,\e,\beta}V_{\nu,\e,\beta}$ equals $\nu\mu_0V_{2,1}$. Writing $V_2 = V_{2,1}+V_{2,r}$ and $v_2 = \nu\mu_0V_{2,1}+v_{2,r}$, we have by \eqref{estimate-T1,2}, \eqref{T2-vs-T1} and \eqref{exp-numu0}
\begin{equation}\label{est-V2}
  V_{2,1} = \cO_{\rm fun}\big(\Theta\sigma^2\big),
  \qquad
  V_{2,r} = \cO_{\rm fun}\Big(\Theta\sigma^2\frac{\e}{\nu}\Big),
  \qquad
  v_{2,r} = \cO_{\rm fun}\Big(\frac{\nu}{\theta}\Theta\sigma^2\frac{\e}{\nu}\Big) .
\end{equation}
Collecting \eqref{VandvSplit}, \eqref{v1-identity}, \eqref{V21} and setting
\begin{equation}\label{V-split}
  V_L \coloneqq V_1 + V_{2,1} , \qquad V_R \coloneqq V_{2,r}+V_r , \qquad v_R \coloneqq v_{2,r}+v_r ,
\end{equation}
we obtain the two expansions
\begin{equation}\label{LV-split}
  V_{\nu,\e,\beta} = V_0 + V_L + V_R ,
  \qquad
  \cL_{\nu,\e,\beta}V_{\nu,\e,\beta} = \nu\mu_0\big(V_0+V_L\big) + v_{\rm lead} + v_R ,
\end{equation}
together with the bounds
\begin{equation}\label{estimate-VL-VR}
  V_L = \cO_{\rm fun}(\Theta\sigma), \qquad
  V_R = \cO_{\rm fun}\Big(\Theta\sigma^2\Big(\sigma+\frac{\e}{\nu}\Big)\Big), \qquad
  v_R = \cO_{\rm fun}\Big(\frac{\nu}{\theta}\Theta\sigma^2\Big(\sigma+\frac{\e}{\nu}\Big)\Big) .
\end{equation}
In particular, one has $\|V_R\|_{L^2}\lesssim\sigma\|V_L\|_{L^2}\lesssim\sigma^2\|V_0\|_{L^2}$ by Lemma \ref{lem:V0} and \eqref{Theta-vs-theta}.

\subsection{Proof of Theorem \ref{thm:main-1}}
We now compute the leading term of the inner product $\langle V_{\nu,\e,\beta}, \cL_{\nu,\e,\beta}V_{\nu,\e,\beta} \rangle_{L^2}$ and prove the second part of Theorem \ref{thm:main-1}. Note that the imaginary part of $\lambda_{\nu,\e,\beta}$ does not affect spectral stability. From \eqref{LV-split},
\begin{align*}
  \langle V_{\nu,\e,\beta}, \cL_{\nu,\e,\beta}V_{\nu,\e,\beta} \rangle_{L^2}
  &= \nu\mu_0 \|V_0+V_L\|_{L^2}^2 + \frac{\e^2}{\nu}\cI_{\nu,\e,\beta}\,\langle V_0+V_L, V_0\rangle_{L^2} \\
  &\quad + \nu\mu_0\langle V_R, V_0+V_L\rangle_{L^2} + \langle V_R, v_{\rm lead}\rangle_{L^2} + \langle V_0+V_L+V_R, v_R\rangle_{L^2} .
\end{align*}
By \eqref{exp-numu0}, \eqref{est-vlead}, \eqref{estimate-VL-VR}, and Lemma \ref{lem:V0}, the three terms on the second line are, respectively,
\[
\cO\Big(\frac{\nu}{\theta}\Theta^2\sigma^2\big(\sigma+\tfrac{\e}{\nu}\big)\Big),
\qquad
\cO\Big(\frac{\e^2\theta}{\nu}\Theta^2\sigma^2\big(\sigma+\tfrac{\e}{\nu}\big)\Big),
\qquad
\cO\Big(\frac{\nu}{\theta}\Theta^2\sigma^2\big(\sigma+\tfrac{\e}{\nu}\big)\Big),
\]
and all three are dominated by the first expression. We now replace $\|V_0+V_L\|^2_{L^2}$ and $\langle V_0+V_L,V_0\rangle_{L^2}$ by $\|V_{\nu,\e,\beta}\|^2_{L^2}$. By \eqref{LV-split} and \eqref{estimate-VL-VR},
\[
\Big| \|V_{\nu,\e,\beta}\|_{L^2}^2 - \|V_0+V_L\|^2_{L^2} \Big| \leq C_U\,\Theta^2\sigma^2\Big(\sigma+\frac{\e}{\nu}\Big) ,
\quad
\Big| \|V_{\nu,\e,\beta}\|_{L^2}^2 - \langle V_0+V_L,V_0\rangle_{L^2} \Big| \leq C_U\,\Theta^2\sigma ,
\]
so the first replacement produces an error of the same type as the three terms above, while the second, which is multiplied by $\frac{\e^2}{\nu}\cI_{\nu,\e,\beta}=\cO(\frac{\e^2\theta}{\nu})$, produces an error $\cO\big(\frac{\e^2\theta}{\nu}\Theta^2\sigma\big)$. Taking real parts and using $\Re(\nu\mu_0)=-\nu\e^2$, we therefore obtain
\begin{equation}\label{exp-LV}
\begin{split}
\Re \big(\langle V_{\nu,\e,\beta}, \cL_{\nu,\e,\beta}V_{\nu,\e,\beta} \rangle_{L^2} \big)
& = \frac{\e^2}{\nu}\Big(-\nu^2  + \Re\, \cI_{\nu,\e,\beta} \Big) \|V_{\nu,\e,\beta}\|_{L^2}^2
  + \cO \Big(\frac{\nu}{\theta}\Theta^2\sigma^2\big(\sigma +\tfrac{\e}{\nu}\big)\Big) \\
& \quad  + \cO \Big(\frac{\e^2\theta}{\nu}\Theta^2\sigma\Big).
\end{split}
\end{equation}
Dividing by $\|V_{\nu,\e,\beta}\|^2_{L^2}$, which by Lemma \ref{lem:V0} and \eqref{estimate-VL-VR} satisfies $\|V_{\nu,\e,\beta}\|_{L^2}\geq\frac12 c_U\Theta$ (If $U \equiv 0$, then $\|V_{\nu,\e,\beta}\|_{L^2} = \e\nu$ and the estimate is trivial), we obtain from \eqref{eq:trivlambda1},
\begin{equation}\label{exp:lambda-real}
  \Re \, \lambda_{\nu,\e,\beta} = \frac{\e^2}{\nu}\Big(-\nu^2  + \Re\, \cI_{\nu,\e,\beta} \Big) + \cO\Big(\frac{\nu}{\theta}\sigma^2\big(\sigma+\tfrac{\e}{\nu}\big)\Big) + \cO\Big(\frac{\e^2\theta}{\nu}\,\sigma\Big).
\end{equation}
It remains to identify the two remainders. Using $\sigma=\frac{\e\theta(1+\beta)}{\nu}$, $\theta\leq1$ and $\sigma+\frac{\e}{\nu}\leq C_U\frac{\e(1+\beta)}{\nu}$, we have,
\[
\frac{\nu}{\theta}\sigma^2\Big(\sigma+\frac{\e}{\nu}\Big)
\leq C_U\,\frac{\nu}{\theta}\cdot\frac{\e^2\theta^2(1+\beta)^2}{\nu^2}\cdot\frac{\e(1+\beta)}{\nu}
= C_U\,\frac{\e^2}{\nu}\cdot\frac{\e}{\nu}\,\theta\,(1+\beta)^3 ,
\]
and similary,
\[
\frac{\e^2\theta}{\nu}\,\sigma = \frac{\e^2}{\nu}\cdot\frac{\e}{\nu}\,\theta^2(1+\beta) \leq \frac{\e^2}{\nu}\cdot\frac{\e}{\nu}\,\theta\,(1+\beta)^3 ,
\]
so that \eqref{exp:lambda-real} becomes
\begin{equation}\label{exp:lambda-real-2}
  \Re \, \lambda_{\nu,\e,\beta} = \frac{\e^2}{\nu}\left(-\nu^2  + \Re\, \cI_{\nu,\e,\beta} + \cO \Big( \theta\,\frac{\e}{\nu}(1+\beta)^3 \Big) \right).
\end{equation}
Finally, we compute $\Re\,\cI_{\nu,\e,\beta}$. Using $(\de_y + \e)^{*} =-\de_y + \e$, and note that
\[
(-\de_{yy}+\e^2)^{-1}U =\partial_{yy}^{-1} U + \cO(\e^2),
\]
then we have
\begin{equation}\label{I}
  \cI_{\nu,\e,\beta}
  = -\langle \partial_{yy}^{-1} U , F_{\nu,\e,\beta} \rangle_{L^2} + \cO\big(\e^2\theta\big),
\end{equation}
and $\cO(\e^2\theta)$ is absorbed into the remainder of \eqref{exp:lambda-real-2}. In Fourier variables, using $U''_j=-j^2U_j$ and $F_{\nu,\e,\beta,j} = \frac{U''_j}{-j^2-\im b}$, we obtain
\[
  \langle \partial_{yy}^{-1} U , F_{\nu,\e,\beta} \rangle_{L^2}
  = \sum_{j\neq0}\overline{\Big(\frac{U_j}{-j^2}\Big)}\,\frac{-j^2U_j}{-j^2-\im b}
  = \sum_{j\neq0}\frac{|U_j|^2}{-j^2-\im b}
  = \sum_{j\neq0}\frac{(-j^2+\im b)|U_j|^2}{j^4+b^2} ,
\]
whence
\begin{equation}\label{real-part-I}
  \Re \;\langle \partial_{yy}^{-1} U , F_{\nu,\e,\beta} \rangle_{L^2} = -\sum_{j\neq0}\frac{j^2|U_j|^2}{j^4+b^2} = \left\langle U, \left(\partial_{y}^4 + \frac{\beta^2}{\e^2 \nu^2}\right)^{-1} U'' \right\rangle_{L^2} .
\end{equation}
Combining \eqref{exp:lambda-real-2}, \eqref{I} and \eqref{real-part-I}, we conclude the second part of Theorem \ref{thm:main-1}.

\bigskip

\section*{Acknowledgments}
R. M. Chen was supported in part by the NSF grant DMS-2205910.
T. Dai was supported in part by NSF grants DMS-2205910 and DMS-2219384.
D. Wang was supported in part by NSF grant DMS-2510532.
The authors are grateful to Rafa\l{} Wrona for providing some helpful comments.

\end{document}